\newif\ifarxiv
\arxivtrue  

\ifarxiv
    \documentclass[letterpaper,11pt]{article}
    \usepackage[margin=1in]{geometry}
    \usepackage[utf8]{inputenc} 
\usepackage[T1]{fontenc}
\makeatletter
\newdimen\@curXheight
\makeatother
\usepackage{tabularx}
\usepackage[colorlinks=true, allcolors=blue]{hyperref}
\usepackage{url}
\usepackage{ifthen}
\usepackage{mdframed}
\usepackage{enumerate}
\usepackage{graphicx}
\usepackage{subfig}
\usepackage{dsfont}
\usepackage{mathtools}
\usepackage{amsmath,amssymb,amsthm}
\usepackage{orcidlink}
\usepackage{bm}
\usepackage[noend]{algpseudocode}
\usepackage{algorithm}
\usepackage{xcolor}
\usepackage{thmtools}
\usepackage{thm-restate}
\usepackage{xspace}
\usepackage{booktabs}
\usepackage{verbatim}
\newcounter{savealgorithm}
\usepackage[title]{appendix}
\usepackage{tikz}
\usetikzlibrary{patterns,decorations.pathreplacing,calc,shapes.misc,math}

\newcommand{\tikzmark}[1]{\tikz[overlay,remember picture] \node (#1) {};}

\newcommand*{\AddNote}[4]{%
    \begin{tikzpicture}[overlay, remember picture]
        \draw [decoration={brace,amplitude=0.5em},decorate, thick,black]
            ($(#3)!(#1.north)!($(#3)-(0,1)$)$) --  
            ($(#3)!(#2.south)!($(#3)-(0,1)$)$)
                node [align=left, text width=2.5cm, pos=0.5, anchor=west] {\hspace{0.2cm}#4};
    \end{tikzpicture}
}%
\newenvironment{subalgorithms}
 {%
  \stepcounter{algorithm}%
  \edef\currentthealgorithm{\thealgorithm}%
  \setcounter{savealgorithm}{\value{algorithm}}%
  \setcounter{algorithm}{0}%
  \renewcommand{\thealgorithm}{\currentthealgorithm\alph{algorithm}}%
 }
 {%
  \setcounter{algorithm}{\value{savealgorithm}}%
 }

\newtheorem{thm}{Theorem}[section]
\newtheorem{lem}[thm]{Lemma}
\newtheorem{prop}[thm]{Proposition}
\newtheorem{claim}[thm]{Claim}
\newtheorem{cor}[thm]{Corollary}
\newtheorem{conj}{Conjecture}[section]

\theoremstyle{definition}
\newtheorem{defn}{Definition}[section]
\newtheorem{assumption}{Assumption}[section]
\newtheorem{rem}{Remark}[section]

\newenvironment{Assumption}
{\begin{mdframed}\begin{assumption}}
		{\end{assumption}\end{mdframed}}

\def\0{\mathbf{0}}
\def\1{\mathbf{1}}

\def\E{\mathbb{E}}

\def\P{\mathbb{P}}
\def\R{\mathbb{R}}

\def\Z{\mathbb{Z}}

\def\bA{\bm{A}}
\def\bN{\bm{N}}

\newcommand{\pin}{p_{\text{in}}}
\newcommand{\pout}{p_{\text{out}}}
\newcommand{\hsigma}{\hat{\sigma}}
\newcommand{\tsigma}{\tilde{\sigma}}
\newcommand{\bsigma}{\bm{\sigma}}
\newcommand{\hbsigma}{\hat{\bsigma}}
\newcommand{\tbsigma}{\tilde{\bsigma}}
\newcommand{\btau}{\bm{\tau}}

\newcommand{\Tmap}{\hat{\Theta}^{\text{MAP}}}

\def\cA{\mathcal{A}}

\def\cD{\mathcal{D}}
\def\cE{\mathcal{E}}
\def\cF{\mathcal{F}}
\def\cG{\mathcal{G}}
\def\cH{\mathcal{H}}
\def\cI{\mathcal{I}}
\def\cL{\mathcal{L}}

\def\cP{\mathcal{P}}

\def\cS{\mathcal{S}}
\def\cT{\mathcal{T}}
\def\cV{\mathcal{V}}

\def\cY{\mathcal{Y}}
\def\cZ{\mathcal{Z}}

\def\A2i{||A^*||_{2\rightarrow\infty}}

\def\torus{(-\frac{n}{2}, \frac{n}{2}]}

\def\weq{ \ = \ }
\def\wge{ \ \ge \ }
\def\wle{ \ \le \ }

\newcommand{\sign}{\operatorname{sgn}}
\newcommand{\Ber}{\operatorname{Ber}}
\newcommand{\Bern}{\operatorname{Bernoulli}}
\newcommand{\Bin}{\operatorname{Bin}}
\newcommand{\Poi}{\operatorname{Poi}}

\newcommand{\remove}[1]{}

\newcommand{\volume}{\operatorname{vol}}

\tikzset{cross/.style={cross out, draw=black, minimum size=2*(#1-\pgflinewidth), inner sep=0pt, outer sep=0pt}, cross/.default={1pt}}

\newcommand{\Erdos}{Erd\H{o}s\xspace}
\newcommand{\Renyi}{R{\'e}nyi\xspace}

\newtheorem{example}{Example}%
\newtheorem{remark}{Remark}%
\newtheorem{definition}{Definition}

\newcommand{\abs}[1]{{\lvert#1\rvert}}

\newcommand{\zeronorm}[1]{{\lVert#1\rVert}_0}
\DeclareMathOperator*{\argmax}{arg\,max}

\numberwithin{equation}{section}

\newcommand{\GKBM}{\operatorname{GKBM}}
\newcommand{\Hel}{\operatorname{Hel}}

\newcommand{\al}{\alpha}
\newcommand{\be}{\beta}
\newcommand{\ga}{\gamma}

\newcommand{\de}{\delta}
\newcommand{\De}{\Delta}
\newcommand{\eps}{\epsilon}
\newcommand{\ka}{\kappa}
\newcommand{\la}{\lambda}
\newcommand{\La}{\Lambda}

\newcommand{\LLupdated}[1]{{\color{purple}#1}}

\newcommand{\Ham}{\operatorname{Ham}}

\newcommand{\map}{\hat{\bsigma}^{\rm MAP}}

\title{Community Detection on Block Models with Geometric Kernels}
\author{Konstantin Avrachenkov\orcidlink{0000-0002-8124-8272}\thanks{INRIA, Sophia Antipolis, 2004 Rte des Lucioles, 06902 Valbonne, France}\and B. R. Vinay Kumar\orcidlink{0000-0002-7329-8659}\thanks{Eindhoven University of Technology, PO Box 513,
5600 MB Eindhoven, The Netherlands.}\and Lasse Leskel\"{a}
\orcidlink{0000-0001-8411-8329}\thanks{Dept. of Mathematics and Systems Analysis, Aalto University, Otakaari 1, 02150 Espoo, Finland}}

\else
    \documentclass[numbers,webpdf,imaiai]{ima-authoring-template}
    \usepackage[utf8]{inputenc} 
\usepackage[T1]{fontenc}
\usepackage{url}
\usepackage{ifthen}
\usepackage{mdframed}
\usepackage{enumerate}
\usepackage{graphicx}
\usepackage{subfig}
\usepackage{dsfont}
\usepackage{mathtools}
\usepackage{amsmath,amssymb,amsthm}
\usepackage{orcidlink_IMA}
\usepackage{bm}
\usepackage{tabularx}
\usepackage{algorithm}
\usepackage{xcolor}
\usepackage{thmtools}
\usepackage{thm-restate}
\usepackage{verbatim}
\newcounter{savealgorithm}
\usepackage[title]{appendix}
\usepackage{tikz}
\usetikzlibrary{patterns,decorations.pathreplacing,calc,shapes.misc,math}

\newcommand{\tikzmark}[1]{\tikz[overlay,remember picture] \node (#1) {};}

\newcommand*{\AddNote}[4]{%
    \begin{tikzpicture}[overlay, remember picture]
        \draw [decoration={brace,amplitude=0.5em},decorate, thick,black]
            ($(#3)!(#1.north)!($(#3)-(0,1)$)$) --  
            ($(#3)!(#2.south)!($(#3)-(0,1)$)$)
                node [align=left, text width=2.5cm, pos=0.5, anchor=west] {\hspace{0.2cm}#4};
    \end{tikzpicture}
}%

\newtheorem{thm}{Theorem}[section]
\newtheorem{lem}[thm]{Lemma}
\newtheorem{prop}[thm]{Proposition}
\newtheorem{claim}[thm]{Claim}

\theoremstyle{definition}
\newtheorem{defn}{Definition}[section]
\newtheorem{assumption}{Assumption}[section]
\newtheorem{rem}{Remark}[section]

\def\0{\mathbf{0}}
\def\1{\mathbf{1}}

\def\E{\mathbb{E}}

\def\P{\mathbb{P}}
\def\R{\mathbb{R}}

\def\Z{\mathbb{Z}}

\def\bA{\bm{A}}
\def\bN{\bm{N}}

\newcommand{\pin}{p_{\text{in}}}
\newcommand{\pout}{p_{\text{out}}}
\newcommand{\hsigma}{\hat{\sigma}}
\newcommand{\tsigma}{\tilde{\sigma}}
\newcommand{\bsigma}{\bm{\sigma}}
\newcommand{\hbsigma}{\hat{\bsigma}}
\newcommand{\tbsigma}{\tilde{\bsigma}}
\newcommand{\btau}{\bm{\tau}}

\newcommand{\Tmap}{\hat{\Theta}^{\text{MAP}}}

\def\cA{\mathcal{A}}

\def\cD{\mathcal{D}}
\def\cE{\mathcal{E}}
\def\cF{\mathcal{F}}
\def\cG{\mathcal{G}}
\def\cH{\mathcal{H}}
\def\cI{\mathcal{I}}
\def\cL{\mathcal{L}}

\def\cP{\mathcal{P}}

\def\cS{\mathcal{S}}
\def\cT{\mathcal{T}}
\def\cV{\mathcal{V}}

\def\cY{\mathcal{Y}}
\def\cZ{\mathcal{Z}}

\def\A2i{||A^*||_{2\rightarrow\infty}}

\def\torus{(-\frac{n}{2}, \frac{n}{2}]}

\def\weq{ \ = \ }
\def\wge{ \ \ge \ }
\def\wle{ \ \le \ }

\newcommand{\sign}{\operatorname{sgn}}
\newcommand{\Ber}{\operatorname{Ber}}
\newcommand{\Bern}{\operatorname{Bernoulli}}
\newcommand{\Bin}{\operatorname{Bin}}
\newcommand{\Poi}{\operatorname{Poi}}

\newcommand{\remove}[1]{}

\newcommand{\volume}{\operatorname{vol}}

\tikzset{cross/.style={cross out, draw=black, minimum size=2*(#1-\pgflinewidth), inner sep=0pt, outer sep=0pt}, cross/.default={1pt}}

\newcommand{\Erdos}{Erd\H{o}s\xspace}
\newcommand{\Renyi}{R{\'e}nyi\xspace}

\newcommand{\abs}[1]{{\lvert#1\rvert}}

\newcommand{\zeronorm}[1]{{\lVert#1\rVert}_0}
\DeclareMathOperator*{\argmax}{arg\,max}

\numberwithin{equation}{section}

\newcommand{\GKBM}{\operatorname{GKBM}}
\newcommand{\Hel}{\operatorname{Hel}}

\newcommand{\al}{\alpha}
\newcommand{\be}{\beta}
\newcommand{\ga}{\gamma}

\newcommand{\de}{\delta}
\newcommand{\De}{\Delta}
\newcommand{\eps}{\epsilon}
\newcommand{\ka}{\kappa}
\newcommand{\la}{\lambda}
\newcommand{\La}{\Lambda}

\newcommand{\LLupdated}[1]{{\color{purple}#1}}

\newcommand{\Ham}{\operatorname{Ham}}

\newcommand{\map}{\hat{\bsigma}^{\rm MAP}}

\DOI{DOI HERE}
\copyrightyear{2026}
\vol{00}
\pubyear{2026}
\access{Advance Access Publication Date: Day Month Year}
\appnotes{Paper}
\copyrightstatement{Published by Oxford University Press on behalf of the Institute of Mathematics and its Applications. All rights reserved.}
\firstpage{1}
\title[Community recovery on GKBM]{Community Detection on Block Models with Geometric Kernels}
\author{
Konstantin Avrachenkov
\ORCID{0000-0002-8124-8272}
\address{
\orgdiv{INRIA}, 
\orgname{Sophia Antipolis}, 
\orgaddress{
\street{2004 Rte des Lucioles}, \state{06902 Valbonne}, \country{France}}}}

\author{B. R. Vinay Kumar*
\ORCID{0000-0002-7329-8659}
\address{
\orgdiv{Eindhoven University of Technology},  
\orgaddress{
\street{P.O. Box 513}, 
\state{5600 MB Eindhoven}, 
\country{The Netherlands}}}}

\author{Lasse Leskel\"{a}
\ORCID{0000-0001-8411-8329}
\address{
\orgdiv{Department of Mathematics and Systems Analysis}, 
\orgname{Aalto University}, 
\orgaddress{
\street{Otakaari 1}, 
\state{02150 Espoo}, 
\country{Finland}}}
\vspace{1ex}}

\authormark{K. Avrachenkov, B. R. Vinay Kumar, L. Leskel\"{a}}

\corresp[*]{Corresponding author: \href{email:vinaykumar.br@iitb.ac.in}{vinaykumar.br@iitb.ac.in}}

\received{Date}{0}{Year}
\revised{Date}{0}{Year}
\accepted{Date}{0}{Year}
\fi

\usepackage{xspace}

\theoremstyle{thmstyletwo}%

\newcommand{\myabstract}{
We introduce the Geometric Kernel Block Model that allows the study of community structures where connection probabilities are influenced by continuous spatial or geometric features,
addressing a limitation of standard block models that ignore observed node attributes.
In this model, every node possesses two independent labels:
an observed location label and a hidden community label.
A geometric kernel maps the locations of pairs of nodes to probabilities,
and edges are drawn based on both their community labels and the value of the kernel corresponding to their locations.
Given a graph so generated along with the vertex location labels,
the latent communities are to be inferred.
In this work, we establish the fundamental statistical limits for recovering the communities in such models.
Additionally, we propose a novel linear-time algorithm (in the number of edges) and show that it recovers the communities of nodes exactly up to the
information-theoretic threshold.
}

\begin{document}

\ifarxiv
    \maketitle
    \begin{abstract}
        \myabstract
    \end{abstract}
\else
    \abstract{\myabstract}
    \maketitle
\fi

\section{Introduction}

Community detection is a fundamental unsupervised learning task with applications in many domains. Its objective is to recover clusters of nodes based on their observed interactions. Stochastic block models (SBMs) provide a widely used generative framework for community-structured networks and have been extensively studied in both theory and practice (e.g.\ \cite{abbe2017community,avrachenkov2022statistical,fortunato202220} and references therein). They can be viewed as \Erdos--\Renyi\ graphs augmented with community structure. For the stochastic block model, the problem of community recovery has been investigated by Mossel, Neeman, and Sly \cite{mossel2012stochastic} in the constant average degree regime, where the authors prove the conditions for impossibility of recovering communities, and in \cite{mossel2018proof} they provide an algorithm to recover when it is indeed possible. Massouli\'{e} in \cite{massoulie2014community} provides a spectral algorithm in the same regime. When the average degree grows logarithmically with the network size, the problem of community detection has been addressed by Abbe, Bandeira, and Hall in \cite{abbe2015exact}. The paper by Abbe  \cite{abbe2017community} provides a comprehensive survey of results on SBMs; see also \cite{avrachenkov2022statistical,fortunato202220} for more recent developments in the field. 

SBMs do not capture the property of transitivity or triadic closure wherein `friends of friends are friends' prevalent in social networks. Similarly, in co-authorship networks, authors of research articles tend to collaborate more with researchers in the same region. The geometric dependence is typically evidenced by the sparsity of long-distance edges, and the abundance of triangles and short-distance edges. Likewise, several methods in image analysis \cite{gao2019geometric} or DNA haplotype reconstruction \cite{sankararaman2020comhapdet} are known to yield better results when mapped into a geometric space. The dependence on geometry is often subtle or hidden in these applications.

Random geometric graphs (RGGs) are a popular model class for spatial data. In these graphs, $N$ nodes are uniformly distributed in a bounded region and edges are placed between two points if they are within a prescribed distance $r$ of each other. Based on the average degree of a (typical) node, RGGs are said to operate in different regimes (see \cite[Chapter 13]{Penrose_2003}). In the \emph{sparse} regime, the average degree is a constant and there are numerous connected components. In the \emph{logarithmic} regime, the average degree grows logarithmic in the size of the network, $N$, and the graph is connected with high probability. Lastly, in the \emph{dense} regime, the average degree grows linearly in the network size. Recent works \cite{abbe2021community,galhotra2023community,avrachenkov2021higher} introduce communities into RGGs and investigate the problem of community detection in the different regimes.

The geometric block model (GBM) analyzed by Galhotra, Mazumdar, Pal, and Saha \cite{galhotra2023community} distributes nodes uniformly at random in a Euclidean unit sphere and connects two nodes of the same (resp.\ different) community when they are within a distance of $r_{\text{in}}$ (resp.\ $r_{\text{out}}$ ) from each other.  Here $r_{\text{in}}>r_{\text{out}}$, and they are chosen so that the RGG operates in the logarithmic regime.
The authors characterize a parameter region where community recovery is impossible. In another region, they provide a triangle-counting algorithm that recovers the communities exactly. However, there is a gap between the two regions where it is not known whether community recovery is possible. In \cite{chien2020active}, the authors Chien, Tulino, and Llorca study the clustering problem on the same model in an active learning setting. It is to be noted here that, in these works, the community recovery algorithm observes only the graph and not the locations of nodes.

Motivated by applications in DNA haplotype assembly \cite{sankararaman2020comhapdet}, Abbe, Baccelli, and Sankararaman in \cite{abbe2021community} propose the Euclidean random graph (ERG) model. Consider a Poisson point process of intensity $\la$ within a box  $\left[-\frac{n^{1/d}}{2},\frac{n^{1/d}}{2}\right]^d $ and communities assigned independently among $\{-1,+1\}$ with equal probability to all nodes. A graph is generated by connecting nodes that are within a prescribed distance $(\log n)^{1/d}$ and with probability either $p$ or $q$ based on whether they are from the same community or from different communities respectively. Here $p>q$. In the logarithmic regime, the authors of \cite{abbe2021community} provided necessary conditions on the parameters for recovering the communities given the graph and the node locations.  They obtain an information quantity 
\begin{equation}
I'(\la,p,q) = 2\la \left[1-\sqrt{pq}-\sqrt{(1-p)(1-q)}\right],
\label{eq:abs_info}
\end{equation}
that governs community recovery. Specifically, they show that if $I'(\la,p,q)<1$, no algorithm can recover the communities exactly and produce an algorithm that can recover the communities when $I'(\la,p,q)>C>1$. However, in the logarithmic regime, the conditions were not tight and the authors conjectured that one could bridge the gap to recover the communities for all possible parameter values. They also suggested an additional refinement step for their algorithm that could remove the gap. In a paper by Gaudio, Niu, and Wei \cite{gaudio2024exact}, the conjecture is resolved in the positive using a novel two-step algorithm. The first step discretizes the space and recovers communities in a small region which is then propagated throughout the space to obtain an initial estimate of the node communities. The second step refines this estimate to recover the true communities exactly. The authors in \cite{gaudio2024exact} show that with a clever choice of discretization, the gap between the necessary and sufficient conditions in \cite{abbe2021community} can indeed be closed. Additionally, their algorithm generalizes to parameter values $p,q$ not necessarily satisfying $p>q$. A subsequent work \cite{gaudio2024exactERG}, introduces the Geometric Hidden Clique Model that encompasses other geometric problems such as the geometric $\Z_2$ synchronization and the geometric submatrix localization.

In this work, we build on this latter body of literature. More specifically, while the ERG model class 
is able to capture applications with a hard spatial threshold, several practical applications involve interactions between points that vary as a function of the distance between them. For example, in a co-authorship network, the frequency of interaction typically follows a spatial hierarchy: researchers in the same city or region interact more often than those geographically distant, but less frequently than those who are in the same institution. Such interactions can be captured using soft random geometric graphs, initially proposed by Penrose in \cite{penrose2016connectivity} wherein a connection function governs the probability of connecting two points given their locations. We introduce community interactions on soft RGGs via the \emph{geometric kernel block model} (GKBM). Instead of possible edges between nodes that are within a prescribed distance from each other as in the ERG model, we introduce a connection function, referred to as a geometric kernel, that outputs a probability of connection between two nodes given their locations. The graph is generated by accounting for this probability along with the node communities, which for two communities is parameterized by $p$ and $q$.

Similar models for community detection on geometric graphs generated via a kernel have been investigated in the sparse regime by Eldan, Mikulincer, and Pieters in \cite{eldan2022community}. However, the authors think of the locations as the communities and provide a spectral algorithm to recover an embedding given the inhomogeneous Erd\H{o}s-\Renyi random graph generated using a rotational invariant kernel. Yet another closely related work is \cite{avrachenkov2021higher} wherein Avrachenkov, Bobu, and Dreveton propose the soft geometric block model where there are two spatial kernels; one for nodes within a community and the other for nodes across communities. The authors use techniques from Fourier analysis to show that higher order eigenvectors recover the communities even when the locations are unknown. However, the analysis there is limited to the dense regime of the RGG. In this work, our interest is in the logarithmic regime.

The main contributions of the present paper are:
\begin{itemize}
\item
Information-theoretic conditions on the GKBM model parameters that guarantee the possibility of exact recovery (existence of a strongly consistent estimator) of node communities for a large class of geometric kernels.
\item
A general analytical framework to obtain tight impossibility results for exact recovery on graphs generated from spatial kernels.
\item
A linear-time algorithm that achieves exact recovery under mild assumptions on the kernel.
\end{itemize}
We restrict ourselves to the case of one-dimensional RGGs in this work, but we believe that most of the techniques carry over to higher dimensions as well. The rest of the paper is organized as follows: Section~\ref{sec:model} describes the GKBM model, and Section~\ref{sec:prob_stat} states the exact recovery problem. The linear-time algorithm and main results are presented in Section \ref{sec:main_results}. The proofs of the impossibility and achievability results are provided in Section~\ref{sec:impossibility} and Section~\ref{sec:achievability}, respectively, with some auxiliary results provided in the appendix. Section~\ref{sec:conclusions} concludes the paper.

\section{Model description}
\label{sec:model}

We study a finite set of nodes $V$ embedded in a circle of circumference~$n$, which we represent as the interval $(-n/2,\,n/2]$ with endpoints identified.
The nodes are characterised by community membership labels $\sigma_v \in \{-1,+1\}$ that are assigned to all $v\in V$ independently and with equal probability. We identify the nodes with their
locations.
Given the community memberships and locations, each undirected node pair
$\{u,v\}$ is linked independently with probability
\begin{equation}
\label{eq:InteractionProduct}
P_{\sigma_u \sigma_v} Q_{uv}
\end{equation}
where 
\begin{equation}
\label{eq:InteractionKernel}
P_{\sigma_u \sigma_v}
=
\begin{cases}
p, &\quad \sigma_u = \sigma_v, \\
q, &\quad \sigma_u \ne \sigma_v,
\end{cases}
\qquad \text{and} \qquad
Q_{uv}
= \phi\bigg( \frac{\| u - v \|}{\log n}\bigg),
\end{equation}
and
$\phi \colon \R_+ \to [0,1]$ is a measurable function of bounded support representing
how interaction probabilities vary with distance
\[
\|x-y\|
\weq \min \{|x-y|, \, n-|x-y|\}.
\]
We refer to $\phi$ as the geometric kernel. The community recovery task amounts to estimating the community membership labels $\{\sigma_v\}$
from the adjacency matrix $\{A_{uv}\}$ of the observed graph and the node locations $V$.

To simplify analysis, we assume that the number of nodes is a Poisson-distributed
random variable with mean $\la n$, which implies that
node configurations restricted to disjoint spatial regions are stochastically independent,
and $\la$ equals the expected node density.
The joint law of $(V, \{\sigma_v\}, \{A_{uv}\})$ is denoted by $\P = \P^{(n)}$
and called the
Geometric Kernel Block Model
with  volume $n$, density $\la$, connection function $\phi$, and baseline intra- and inter-community link rates $p,q$.
The model is abbreviated as $\GKBM_n(\la, \phi, p, q)$.
This model smoothly interpolates between
soft geometric random graphs \cite{penrose2016connectivity} and
the standard stochastic block model \cite{abbe2017community},
with the former corresponding to $P_{\sigma \sigma'} = 1$,
and the latter to $Q_{uu'} = 1$ in \eqref{eq:InteractionProduct}.
The normalising factor $\log n$ in \eqref{eq:InteractionKernel} is chosen
so that the average degree in the graph is $\Theta(\log n)$,
which is the critical regime for the connectivity of soft random geometric graphs \cite{penrose2016connectivity,wilsherConnectivityOnedimensionalSoft2020,Wilsher_Dettmann_Ganesh_2023},
and for the exact recovery in standard stochastic block models
\cite{abbe2015exact,Mossel_Neeman_Sly_2016}.

\textbf{Notation:} $|C|$ denotes the cardinality of a set $C$.
$V(C) = V \cap C$ denotes the set of points in $C$, for a node configuration $V$.
A set $C$ is called \emph{$\delta$-occupied} if $|V(C)| \ge \delta \log n$. The Lebesgue measure of a set $C$ is denoted by $\volume(C)$.
Vectors and matrices are denoted using boldface symbols.
For example, $\bsigma=(\sigma_u)_{u \in V}$ and $\bA = (A_{uv})_{u,v\in V}$. Note that the variables $(V,\sigma_v,A_{uv})$ are all dependent on $n$.
When it is necessary to make this explicit, we use $V^{(n)}, \bsigma^{(n)}, \bA^{(n)}$. The notation $\P_V$ is the distribution of $(\bsigma,\bA)$ conditioned on $V$.
We denote
\[
\sign(x) = \begin{cases}
    +1 & \text{ if } x\ge 0, \\
    -1 & \text{ otherwise}.
\end{cases}
\]

\section{Problem statement}
\label{sec:prob_stat}
We study the unsupervised machine learning task of
recovering the community labels $\bsigma^{(n)}$ given the adjacency matrix $\bA^{(n)}$ and the location labels $V^{(n)}$.
For an estimator $\hat{\bsigma}^{(n)} = \hat{\bsigma}^{(n)}(\bA^{(n)}, V^{(n)})$, we define its permutation-invariant
Hamming distance
to
the ground-truth
community labels $\bsigma^{(n)}$
by
\begin{equation}
 \Ham(\hbsigma^{(n)},\bsigma^{(n)})
 = \min_{s \in \{\pm 1\}} |\{v\in V^{(n)} \colon \hsigma_v \ne s \sigma_v\}|.
 \label{eq:Hamming_dist}
\end{equation}
The minimum accounts for the fact that, given the node locations and the graph structure, the community labels are identifiable only up to a global flip.

An estimator is said to recover the community structure \emph{exactly} if 
\begin{equation}
\lim_{n \to \infty} \P \left( \frac{\Ham(\hbsigma^{(n)},\bsigma^{(n)})}{n} = 0 \right) \weq 1,
\label{eq:exact_recovery}
\end{equation}
and \emph{almost exactly} if
\begin{equation}
\lim_{n \to \infty} \P\left( \frac{\Ham(\hbsigma^{(n)},\bsigma^{(n)})}{n} < \eta \right)
\weq 1
\qquad \text{for every $\eta >0$}.
\label{eq:almost_exact_recovery}
\end{equation}
In this study we focus on the exact recovery task, aiming to characterise for which combinations of model parameters $(\la, \phi, p, q)$ exact recovery is possible in large networks with $n \gg 1$, and to identify a fast algorithm capable of performing this task. Our algorithm initially recovers the communities almost exactly, and refines the obtained estimate to exactly recover them. 

While previous works \cite{abbe2021community,gaudio2024exact,gaudio2024exactERG} investigate the exact recovery problem with a hard threshold geometric kernel $\phi(x) = \mathds{1}\{x\in [0,1]\}$, in the present paper we allow for a wide range of geometric kernels. We first show an impossibility result by obtaining an information-theoretic threshold,
below which no algorithm can recover the communities exactly.
On the algorithmic side, we provide an algorithm that can recover the communities exactly up to the information-theoretic threshold.
Our work builds on the algorithm in \cite{gaudio2024exact} and adapts it to general geometric kernels. Techniques such as neighbour counting do not suffice since they cannot capture the dependence with the distance.
Our algorithm initially recovers the communities exactly within a small block,
and propagates it using a function of the recovered communities with
distance-dependent
weights.
In addition, we also show matching lower bounds governed by information quantities akin to \eqref{eq:abs_info}. Our results are summarized in the next section.

\section{Main results}
\label{sec:main_results}

To state our main results, we define an information quantity 
\begin{equation}
I_{\phi}(p,q)
\ := \ 2 \int_{\R_+} \bigg( 1- \sqrt{pq} \phi(x) - \sqrt{(1-p\phi(x))(1-q\phi(x))} \bigg) \, dx
\label{eq:info_metric}
\end{equation}
and an interaction range
\begin{equation}
\zeronorm{\phi}
\ := \
\sup \{x \ge 0 \colon \phi(x) \ne 0\}.
\label{eq:kappa}
\end{equation}
The following theorem provides conditions on the model parameters for which the node communities cannot be recovered exactly.

\begin{thm}
\label{thm:impossibility}
If $\la \zeronorm{\phi} < 1$ or $\la I_\phi(p,q) < 1$,
then no estimator can exactly recover the communities
in the $\GKBM_n(\la, \phi, p, q)$ model.
\end{thm}

On the other hand, when the model parameters do not lie in the regime described in Theorem~\ref{thm:impossibility}, we provide an algorithm for community recovery detailed in Algorithm~\ref{alg:full}, and show that with the appropriate initialization it can recover the communities exactly for kernels that are bounded away from zero within the support. More formally, we have the following theorem for community recovery.

\begin{thm}
\label{thm:exact_recovery}
If $\la \zeronorm{\phi} > 1$ and $\la I_\phi(p,q) > 1$,
and if $\phi(x) > 0$ for all $x \le \zeronorm{\phi}$,
then
Algorithm~\ref{alg:full}
with parameters $\chi$ and $\delta$ chosen according to \eqref{eq:chi}--\eqref{eq:delta}
exactly recovers the communities in the $\GKBM_n(\la, \phi, p, q)$ model.
\end{thm}

\makebox[\linewidth]{%
\begin{minipage}{0.9\linewidth}
\begin{algorithm}[H]
\algrenewcommand\algorithmicrequire{\textbf{Input:}}
\algrenewcommand\algorithmicensure{\textbf{Output:}}
\caption{Exact recovery in the GKBM}

\begin{algorithmic}[1]

\Require
Node set $V \subset \torus$, adjacency matrix $\{A_{uv}\} \in \{0,1\}^{|V| \times |V|}$,
model parameters $\la, \phi, p, q$,
tuning parameters $\chi,\delta > 0$.

\Ensure Community membership vector 
$\{\hsigma_v\} \in \{-1,+1\}^{|V|}$

\State Partition $(-\frac{n}{2}, \frac{n}{2}]$ into segments of length $\chi \log n$

\label{line:block_division}

\State Let $B_1, \dots, B_J$ be the segments that contain at least $\delta \log n$ nodes, in the clockwise order

\State Assign $V_j \gets V \cap B_j$ for $j=1,\dots,J$ \label{line:shorthand}

\State Assign $Q_{uv} \gets \phi\Big( \frac{\| u - v \| }{\log n} \Big)$ for $u,v\in V$
\label{line:Qmatrix}

\State Choose an arbitrary reference node $u_0 \in V_1$ and set $\tsigma_{u_0} \gets +1$ \label{line:choose_u0}\tikzmark{top}

\For{$u \in V_1 \backslash \{u_0\}$}
\State $M(u,u_0,B_1) \gets \frac{(p+q)^2}{4}
\sum_{v\in V_1 \setminus \{u,u_0\}} Q_{uv}Q_{u_0v}$
\State $N_{u_0,u} \gets \sum_{v \in V_1} A_{u_0v} A_{uv}$
\label{line:count_common_nbrs}
\State Assign $\tsigma_u \gets +1$ if $N_{u_0,u} > M(u,u_0,B_1)$ and $\tsigma_u \gets -1$ otherwise \tikzmark{right}
\label{line:init_compare}
\EndFor
\tikzmark{bottom}

\For {$j=1, \dots, J-1$} \label{line:iter_start} \tikzmark{Prop_top}
\For{$u \in V_{j+1}$}
\State
$\tsigma_u \gets \sign \Big( \sum_{v \in V_{j}} \tsigma_v
\left[ A_{uv} \log \frac{p}{q} + (1-A_{uv}) \log \frac{1-p Q_{uv}}{1-q Q_{uv}} \right]
\Big)$
\label{line:prop_assignment}
\EndFor

\EndFor \label{line:iter_end} 

\State Assign $\tsigma_u \gets 0$ for $u \in V\setminus  \cup_{j} V_j$ 
\label{line:after_prop}\tikzmark{Prop_bottom}

\For{ $u \in V$} \tikzmark{Refine_top} \label{line:refine_start}
\State $\hsigma_u\gets \sign \Big( \sum_{v} \tsigma_v \left[
A_{uv} \log \frac{p}{q} + (1-A_{uv})
\log \frac{1-p Q_{uv}}{1-q Q_{uv}}
\right] \Big)$
\EndFor \label{line:refine_end}
\tikzmark{Refine_bottom}
\end{algorithmic}
\AddNote{top}{bottom}{right}{ Initialization}
\AddNote{Prop_top}{Prop_bottom}{right}{Propagation}
\AddNote{Refine_top}{Refine_bottom}{right}{Refinement}
\label{alg:full}
\end{algorithm}
\end{minipage}}

\vspace{1em}

Algorithm~\ref{alg:full} requires tuning parameters $\chi,\delta > 0$ as input.
The parameter $\chi$ sets the baseline resolution,
as the algorithm starts by dividing\footnote{One of the segments would have a size less than $\chi \log n$ and we take it to be not $\delta$-occupied. We work with the number of segments being $\frac{n}{\chi \log n}$ instead of $\lceil\frac{n}{\chi \log n}\rceil$. This does not affect the analysis. Here, $\lceil r \rceil$ denotes the smallest integer greater than or equal to $r$. }
the circle into segments of length $\chi \log n$.
The parameter $\delta$ is a threshold parameter for selecting dense
segments among the baseline segments, along which the algorithm propagates.
When proving Theorem~\ref{thm:exact_recovery},
we assume that these parameters satisfy
\begin{equation}
 0 < \chi < \frac{\la \zeronorm{\phi}-1}{2\la}.
 \label{eq:chi}
\end{equation}
and
\begin{equation}
 \label{eq:delta}
 0
 < \delta
 < \la \chi \left(1 - 2 \frac{\chi}{\zeronorm{\phi}} \right) h^{-1} \left( \frac12 + \frac{1}{2\la(\zeronorm{\phi} - 2\chi)} \right),
\end{equation}
where $h^{-1}(\cdot)$ is the inverse of $h(x) = x \log x + 1 - x$ on $(0,1)$.
These choices enable the algorithm to run on $\delta$-occupied segments in the subsequent steps,
thus enabling community recovery.
(Recall that a segment $B$ is \emph{$\delta$-occupied} if
$|V(B)| \ge \delta \log n$, where $V(B) = V \cap B$
denotes the set of nodes in $B$.)

\begin{figure}
\centering
\begin{tikzpicture}
\draw[thick] (0,0) circle(3cm);

\foreach \i in {0,15,30,...,345} {
    \draw[thick] (\i:3cm) -- (\i:2.7cm);
}

\def\denseAnglesA{0}
\def\denseAnglesB{30}
\def\denseAnglesC{60}
\def\denseAnglesD{75}
\def\denseAnglesE{90}
\def\denseAnglesF{135}
\def\denseAnglesG{165}
\def\denseAnglesH{180}
\def\denseAnglesI{225}
\def\denseAnglesJ{255}
\def\denseAnglesK{270}
\def\denseAnglesL{300}
\def\denseAnglesM{315}
\def\denseAnglesN{345}

\def\denseList{{0,30,60,75,90,135,165,180,225,255,270,300,315,345}}

\pgfmathsetseed{1991}

\foreach \i in {1,...,200} {
    \pgfmathsetmacro{\bias}{int(rand*100)}

    \ifnum\bias<70
        \pgfmathsetmacro{\blockIndex}{int(rand*14)}
        \edef\angle{
            \ifnum\blockIndex=0 0\else
            \ifnum\blockIndex=1 30\else
            \ifnum\blockIndex=2 60\else
            \ifnum\blockIndex=3 75\else
            \ifnum\blockIndex=4 90\else
            \ifnum\blockIndex=5 135\else
            \ifnum\blockIndex=6 165\else
            \ifnum\blockIndex=7 180\else
            \ifnum\blockIndex=8 225\else
            \ifnum\blockIndex=9 255\else
            \ifnum\blockIndex=10 270\else
            \ifnum\blockIndex=11 300\else
            \ifnum\blockIndex=12 315\else
            \ifnum\blockIndex=13
            345\fi\fi\fi\fi\fi\fi\fi\fi\fi\fi\fi\fi\fi\fi
        }
        \pgfmathsetmacro{\angle}{\angle + (rand+1)*7.5}
    \else
        \pgfmathsetmacro{\angle}{rand*360}
    \fi

    \pgfmathsetmacro{\x}{3*cos(\angle)}
    \pgfmathsetmacro{\y}{3*sin(\angle)}

    \pgfmathsetmacro{\cval}{int(rand*2)}
    \ifnum\cval=0
        \def\pointcolor{red}
    \else
        \def\pointcolor{blue}
    \fi

    \fill[\pointcolor] (\x,\y) circle (1.5pt);
}

\node at (97.5:2.4cm) {\( B_1 \)};
\node at (82.5:2.4cm) {\( B_2 \)};
\node at (67.5:2.4cm) {\( B_3 \)};

\node at (142.5:2.4cm) {\( B_J \)};
\node at (172.5:2.4cm) {\( B_{J-1} \)};

\def\denseCenters{{ 360, 345, 315, 300, 270, 255, 225, 180, 165, 135, 90, 75, 60, 30, 0 }}
\foreach \i [evaluate=\i as \aStart using {\denseCenters[\i]+7.5}, 
              evaluate=\i as \aEnd using {\denseCenters[\i+1]+7.5}] in {0,...,13} {
    \path coordinate (Pstart) at (\aStart:3cm);
    \path coordinate (Pend) at (\aEnd:3cm);

    \pgfmathsetmacro{\xoffset}{cos((\aStart+\aEnd)/2))}
    \pgfmathsetmacro{\yoffset}{sin((\aStart+\aEnd)/2))}

   \draw[thick,->] (Pstart) .. controls ($(Pstart)!0.5!(Pend) + (\xoffset,\yoffset)$) .. (Pend);
}
\draw[decorate, decoration={brace, amplitude=5pt,mirror}]
    (330:2.5cm) arc[start angle=330, end angle=315, radius=2.5cm];
    
\node at (322.5:1.8cm) {$\chi \log n$};

\end{tikzpicture}
\begin{minipage}{0.9\textwidth}
\caption{Division of $(-\frac{n}{2}, \frac{n}{2}]$ into segments of length $\chi \log n$. The $\delta$-occupied segments are denoted $B_j, j = 1,\cdots,J$.}
\label{fig:block_division}
\end{minipage}
\end{figure}

\begin{figure}
\centering
\begin{tikzpicture}

\draw[thick] (0, 0) -- (14, 0); 

\foreach \x in {0, 7, 14} {
    \draw[thick] (\x, 0) -- (\x, 0.2); 
}

%

\foreach \x in {0.5, 1.8, 3.2, 4.1, 5.3, 6.1, 2.3, 3.9, 3.3, 4.4} {
        \filldraw[red] (\x,0) circle (1mm);
}

\foreach \x in {1.1, 1.4, 2.1, 4.3, 5.5, 6.8, 2.7, 4, 3.1, 5.7} {
       \filldraw[blue] (\x,0) circle (1mm);
}

\fill[color=gray] (9.4, -0.1) rectangle (9.6, 0.1); 
\draw[color=black] (9.4, -0.1) rectangle (9.6, 0.1); 

\foreach \x in {0.5, 1.4, 2.1, 3.2, 4.1, 5.3, 6.1, 6.8, 4.4, 5.7} {
    \pgfmathsetmacro{\midpoint}{(\x + 9.5) / 2}

    
    \draw[black,->] 
        (\x , 0) .. controls (\midpoint, 1) .. (9.5, 0); 
}

\node at (3.5cm,-0.5cm) {\( B_j \)};
\node at (10.5cm,-0.5cm) {\( B_{j+1} \)};
\node at (9.5,-0.25) {\(u\)};
\end{tikzpicture}
\caption{Illustration of the propagation step to a subsequent segment.}
\label{fig:pro}
\end{figure}

The algorithm is divided into three phases: Initialization, Propagation and Refinement. The Initialization phase recovers communities within a single segment and runs in $O(\log^2 n)$ time.
Next, the Propagation phase evaluates a sum over (at most) the nodes in the previous segment for every node in the current segment (see Fig.~\ref{fig:pro}) and repeats this computation over $\delta$-occupied segments as shown in Fig.~\ref{fig:block_division},
yielding a runtime of $O(n \log n$).
Finally, the Refinement phase can run up to $O(n^2$) time, since the coefficients $Q_{uv}$ have to be evaluated for every pair of nodes in Line~\ref{line:Qmatrix}.
However, since the neighbourhood of every node contains $O(\log n$) nodes in the GKBM model when $\phi$ has a bounded support,
using more economical data structures (such as, adjacency lists) the computation of the constants $Q_{uv}$ and therefore the running time of the Refinement phase can be improved to $O(n\log n)$.
We conclude that Algorithm~\ref{alg:full} recovers the communities exactly in the GKBM model in $O(n \log n$) time, which is linear in the number of edges.

\begin{rem}
The key information quantity $I_\phi(p,q)$ 
appearing in Theorems~\ref{thm:impossibility} and~\ref{thm:exact_recovery}
can be interpreted as follows.
By writing
$
2( 1 - \sqrt{xy} - \sqrt{(1-x)(1-y)})
= (\sqrt{x}-\sqrt{y})^2 +  (\sqrt{1-x}-\sqrt{1-y})^2,
$
we see that
$I_\phi(p,q) = T_{1/2}( p\phi \, \| \, q\phi ) + T_{1/2}( 1-p\phi \, \| \, 1-q\phi )$
where
$T_{1/2}(f \| g) = \int_0^\infty ( \sqrt{f(x)} - \sqrt{g(x)} )^2 \, dx$
is the Tsallis divergence of order 1/2 between sigma-finite measures $f(x) dx$ and $g(x)dx$.
Because \Renyi divergences tensorise over product measures,
and \Renyi divergences between Poisson point pattern laws
are given by the Tsallis divergences between the associated intensity measures
\cite[Theorem 5]{Leskela_2024},
it follows that
\[
 I_\phi(p,q)
 \weq D_{1/2}( \cP_{p\phi} \otimes \cP_{1-p\phi} \, , \, \cP_{q\phi} \otimes \cP_{1-q\phi} ),
\]
where $D_{1/2}$ refers to the \Renyi divergence of order 1/2,
and $\cP_f$ denotes the law of a Poisson point pattern on $\R_+$
with intensity function $f$, and $\otimes$ indicates the product of probability measures.
\end{rem}

\section{Proof of impossibility}
\label{sec:impossibility}

This section provides 
the proof of Theorem~\ref{thm:impossibility}. Section \ref{sec:connectivity} justifies the condition $\la \zeronorm{\phi} <1$ for the impossibility of recovering communities by alluding to the connectivity of the underlying graph. In Section \ref{sec:info_metric}, we show that the condition $\la I_\phi(p,q) <1$ is the information-theoretic criterion that characterizes the inability to recover the two communities. Section \ref{sec:proof_impossibility} brings everything together to prove Theorem~\ref{thm:impossibility}.

\subsection{Connectivity criterion for community recovery}
\label{sec:connectivity}
To analyse connectivity, we may couple the model 
$\GKBM_n(\la, \phi, p, q)$ with $\GKBM_n(\la, \phi, 1, 1)$
as follows:
\begin{enumerate}
\item Sample $( V,\{\sigma_v\},\{A'_{uv}\})$ from $\GKBM_n(\la, \phi, 1, 1)$.
\item Sample a symmetric random matrix $\{A''_{uv}\}$ with independent upper triangular entries
so that $A''_{uv}=1$ with probability $P_{\sigma_u \sigma_v}$ as in \eqref{eq:InteractionKernel}.
\item Let $A_{uv} = A'_{uv} A''_{uv}$.
\end{enumerate}
Then $(V,\{\sigma_v\},\{A_{uv}\})$ is distributed according to $\GKBM_n(\la, \phi, p, q)$,
and the graph $G$ with adjacency matrix $\{A_{uv}\}$ is an edge-percolated version of
the graph $G'$ with adjacency matrix $\{A'_{uv}\}$.  In particular, $G$ is a subgraph of $G'$.
Furthermore, we note that $(V, \{A'_{uv}\})$ is an instance of a soft random geometric graph
\cite{penrose2016connectivity,Wilsher_Dettmann_Ganesh_2023}.

In \cite{Wilsher_Dettmann_Ganesh_2023}, the authors show that if $\la \|\phi\|_1 <\frac{1}{2}$, then there exists at least one isolated node. Here $\|\phi\|_1 = \int_0^\infty \phi(x)dx \le  \zeronorm{\phi}$ for kernels with a bounded support. The condition $\la \|\phi\|_1 <\frac{1}{2}$ characterizes the graphs that have an isolated node, and the condition $\la \zeronorm{\phi} <1$ provides a sufficient condition for the graph to be disconnected. The former is more restrictive as compared to the latter, since a graph could be disconnected without having an isolated node. The reason for disconnection is uncrossed gaps in one dimension \cite{Wilsher_Dettmann_Ganesh_2023} as opposed to isolated nodes which are prevalent in higher dimensions \cite{penrose2016connectivity}.

\begin{lem}
If $\la \zeronorm{\phi} < 1$, then the graph $G'$ sampled from 
$\GKBM_n(\la, \phi, 1, 1)$ is disconnected with high probability as $n \to \infty$.
\label{lem:connectivity_softrgg}
\end{lem}
\begin{proof}
Denote $\ka = \zeronorm{\phi}$.
Divide the space $(-\frac{n}{2}, \frac{n}{2}]$ into segments $D_i$ for $i=1,\cdots, \lceil\frac{n}{\ka \log n}\rceil$ of length $\ka \log n$ each. Denote the number of segments by $b = \lceil\frac{n}{\ka \log n}\rceil$. 
Notice that there are no edges possible between non-adjacent segments since the support of the kernel is at most $\ka \log n$. Thus, if two empty segments are non-adjacent with non-empty segments between them, the graph $G'$ has at least two disjoint connected components.

Denote by $\ga$ the probability that
a particular segment $D_i$ is empty.
Because the number of points in $D_i$ is Poisson-distributed with mean $\la \ka \log n$, we find that
\begin{equation}
 \ga
 \weq e^{-\la \ka \log n}
 \weq n^{-\la \ka}.
 \label{eq:empty_block_prob}
\end{equation}
Let $\cD$ be the event that there are at least two empty segments that are non-adjacent and separated by (at least) a non-empty segment. Let $\cY_k$ be the event of having exactly $k$ empty segments with at least two empty non-adjacent segments and separated by a non-empty segment.
Then
\begin{align*}
\P(\cD) & \weq \sum_{k=2}^{b-1} \P(\cY_k) \weq \sum_{k=2}^{b-1} \Bigg(\binom{b}{k} -b\Bigg) \ga^k \big(1-\ga\big)^{b-k} \\
&\wge \sum_{k=1}^{b} \binom{b}{k} \ga^k \big(1-\ga \big)^{b-k} - b\big(1-\ga\big)^{b}\sum_{k=1}^{b} \ga^k \big(1-\ga\big)^{-k}\\
&\wge 1-(1-\ga )^{b} -\frac{b\ga}{1-2\ga}(1-\ga )^{b},
\end{align*}
where the last step is obtained by evaluating the binomial and geometric sums. Since $\ga = n^{-\la \ka } \le \frac{1}{4}$ for sufficiently large $n$, we have that $\frac{1}{1-2\ga} \le 2$. Then, we obtain  
\begin{align*}
\P(\cD)
&\ge 1-(1-\ga)^{b} (1+2b\ga)\\
& \ge  1- e^{-\ga b}(1+2b\ga)\\
&\ge 1-e^{-\frac{n^{1-\la \ka}}{\ka \log n}} \Big[1+ \frac{2n^{1-\la \ka}}{\ka \log n}\Big].
\end{align*}
If $\la \ka <1$, then $\P(\cD) \rightarrow 1$ as $n\rightarrow \infty$.  
\end{proof}

\subsection{Information-theoretic criterion for cluster separation}
\label{sec:info_metric}
We begin this subsection by providing some preliminaries on constructing Palm versions of the $\GKBM_n(\la,\phi,p,q)$ model in Section \ref{sec:palm}. Section \ref{sec:map_analysis} analyzes the Maximum-A-Posteriori (MAP) estimate of the ground-truth communities and establishes conditions for it to fail. The conditions are in terms of the first and second moment of a random variable which are analyzed in Section \ref{sec:first_moment} and Section \ref{sec:second_moment} respectively.

\subsubsection{Palm versions and probabilities}
\label{sec:palm}

\begin{defn}
The Palm version of the $\GKBM_n(\la,\phi,p,q)$ model given points $x_1,\dots,x_r$ in the interval $\torus$ is generated using the following procedure:
\begin{enumerate}
    \item Sample a finite node set $V \subset \torus$ from a homogeneous Poisson point pattern with intensity~$\la$. 
    \item Define $V^{x_1,\dots,x_r} = V \cup \{x_1,\dots,x_r\}$.
    \item Assign each $v \in V^{x_1,\dots,x_r}$ a community membership label $\sigma_v \in \{-1,+1\}$ uniformly at random.
    \item Sample a symmetric random matrix $(A_{uv})_{u,v \in V^{x_1,\dots,x_r}}$ with independent entries above the diagonal sampled from the Bernoulli distribution with success probability $P_{\sigma_u \sigma_v} Q_{uv}$, where $P,Q$ are defined in \eqref{eq:InteractionKernel}. 
\end{enumerate}
The triple $(V^{x_1,\dots,x_r}, (\sigma_v)_{v \in V^{x_1,\dots,x_r}}, (A_{uv})_{u,v \in V^{x_1,\dots,x_r}})$ is a sample from the $\GKBM_n^{x_1,\dots,x_r}(\la,\phi,p,q)$ model. The corresponding probability measure, referred to as the Palm probability, is denoted as $\P^{x_1,\cdots,x_r}$ and the expectation with respect to it is denoted using $\E^{x_1,\cdots,x_r}$.
\end{defn}

\subsubsection{Maximum-A-Posteriori (MAP) estimate} 
\label{sec:map_analysis}
For a finite node set $V \subset \torus$, let $\P_{V}(\cdot)$ denote the distribution of $(\bsigma,\bA)$ conditioned on the locations $V$. Define the MAP estimate of the node communities as 
\begin{equation}
    \map = \argmax_{\bsigma'}\ \P_{V}(\bsigma' \, | \, \bA),
    \label{eq:MAP}
\end{equation}
where ties are broken arbitrarily.
The MAP estimate is Bayes optimal in the sense that
\begin{equation}
\P_{V}( \map \not\in \{\bsigma,-\bsigma\} ) \weq \inf_{t \in \cA(V,\bA)} \P_V(t(V,\bA) \not\in \{\bsigma,-\bsigma\}),
\label{eq:MAP_Bayes_opt}
\end{equation}
where $\cA(V,\bA)$ is the set of all measurable functions of $V$ and $\bA$ (see Appendix \ref{app:MAP_Bayes_opt}). In particular, if there exists an estimate that can recover the ground-truth communities exactly, then the MAP estimate recovers the communities exactly. However, if the MAP estimate in \eqref{eq:MAP} is not unique, or not equal to the ground-truth community vector $\bsigma$ up to a global sign flip, then there is no hope to recover the communities exactly. Thus, in order to obtain conditions when community recovery is not possible, it suffices to show that the MAP estimate is not unique. In the following, we introduce a few notations and terminologies that will be useful to analyze the MAP estimate.
\begin{defn}[Visibility set]
\label{defn:visibility_set}
For a node $u \in V$, its visibility set
\[
\cV(u) := \{v \in V\setminus \{u\}  : \|v-u\| \le \zeronorm{\phi} \log n\}.
\]
\end{defn}
\noindent Let $\Ber(p)(x) = p^x(1-p)^{1-x}$. Define 
\[
\cL_u(k,\bsigma_{\sim u},V,\bA) := \sum_{v\in \cV(u)}\log \big(\Ber(P_{k,\sigma_v}Q_{uv})(A_{uv})\big),
\]
the \emph{log-likelihood} of the community membership of node $u$ relative to the community membership $\bsigma_{\sim u}:= \{\sigma_v: v\in V \setminus \{u\}\}$ of the other nodes and the adjacency matrix $\bA$. Note that it suffices to restrict the sum to the nodes in the visibility set of $u$ since the kernel $\phi$ has a bounded support. For any $u\in V$, define the event
\begin{equation}
\cE_{u}^V
=\bigg\{(\bsigma,\bA) \in \{\pm1\}^V \times \{0,1\}^{V\times V} : \frac{\cL_u(-\sigma_u,\bsigma_{\sim u},V,\bA)}{\cL_u(\sigma_u,\bsigma_{\sim u},V,\bA)}\ge 1\bigg\}, \quad \text{ and let } \quad \xi_u^V = \1\{(\bsigma,\bA)\in \cE_u^V\}.
\label{eq:ubad}
\end{equation}
The following lemma provides a sufficient condition for the non-uniqueness of the MAP estimate.
\begin{lem}
\label{lem:compMAPnec}
Let $\cE := \cup_{u\in V} \cE_u^V$, where $\cE_u^V$ is defined in \eqref{eq:ubad}. Then 
\[
\P_{V}(\cE) \wle \P_{V}( \map \text{ is not unique up to a global sign flip or }\map \not \in \{\bsigma,-\bsigma\}).
\]
\end{lem}
\begin{proof}
Firstly, note that the event $\cE_u^V$ can equivalently be written as
\[
\cE_{u}^V=\bigg\{(\bsigma,\bA) \in \{\pm1\}^V \times \{0,1\}^V : \frac{\P_V( -\sigma_u|\bA,\bsigma_{\sim u})}{\P_V( \sigma_u|\bA,\bsigma_{\sim u})}\ge 1\bigg\}.
\]
Indeed, using Bayes' theorem, it holds that 
\[
\P_V( \sigma_u'|\bA,\bsigma_{\sim u})=   \frac{\P_V( \bA|\sigma_u',\bsigma_{\sim u}) \P_V(\sigma_u'|\bsigma_{\sim u})}{\P_V(\bA|\bsigma_{\sim u})} = \frac{\P_V( \bA|\sigma_u',\bsigma_{\sim u})}{2 \,\P_V(\bA|\bsigma_{\sim u})}.
\]
Since $\log \P_V( \bA|\sigma_u',\bsigma_{\sim u}) = \sum_{v\in \cV(u)} \log \Ber (P_{u,v} Q_{\sigma_u',\sigma_v})(A_{uv})$, the condition $\log \P_V( -k|\bA,\bsigma_{\sim u}) \ge \log \P_V( k|\bA,\bsigma_{\sim u})$ is the same as 
$\frac{\cL_u(-\sigma_u,\bsigma_{\sim u},V,\bA)}{\cL_u(\sigma_u,\bsigma_{\sim u},V,\bA)}\ge 1$.
Therefore,
\begin{align*}
\cE
&= \{(\bsigma,\bA): \exists u \text{ such that }
\P_{V}(-\sigma_u \ | \ \bA,\bsigma_{\sim u}) \ge  \P_{V}(\sigma_u \ | \ \bA,\bsigma_{\sim u}) \} \\
&\subseteq \{(\bsigma,\bA): \exists \bar{\bsigma}\not \in \{\bsigma,-\bsigma\} \text{ such that }
\P_{V}(\bar{\bsigma} \ | \ \bA) \ge  \P_{V}(\bsigma \ | \ \bA) \} \\
&= \{(\bsigma,\bA): \map \text{ is not unique up to a global sign flip or } \map \not \in \{\bsigma,-\bsigma\} \}.
\end{align*}
The second step above is obtained by taking
$\bar{\bsigma} = (-\sigma_u,\bsigma_{\sim u})$. This concludes the proof of the lemma.
\end{proof}
Let $Z = \sum_{u\in V} \xi_u^V$. For $(V,\bsigma,\bA)$ sampled from $\GKBM_n(\la,\phi,p,q)$ model, we say that node $u$ is bad if $\xi_u^V=1$. From Lemma~\ref{lem:compMAPnec}, it is clear that there is no unique MAP estimate if $Z \ge 1$. The following lemma provides conditions when there exists at least one bad node. 
\begin{lem}
\label{lem:first_second_moment_cond}
Let $\cE_0 := \{(V,\bsigma,\bA): (\bsigma,\bA_{0:})\in \cE_0^V\}$. If 
\begin{equation}
\limsup_{n \to \infty} \frac{ \int_{ (-\frac{n}{2}, \frac{n}{2}]}
\E^{0,y} \Big[ \xi^{0y}_0 \xi^{0y}_y \Big] dy}{n \P^0(\cE_0)^2}
\wle 1,
\label{eq:secondmomentcond}
\end{equation}
and
\begin{equation}
\lim_{n \to \infty} n \P^0(\cE_0)
\weq \infty,
\label{eq:firstmomentcond}
\end{equation}
then there exists at least one bad node i.e., $Z\ge 1$ with high probability.
\end{lem}
\begin{proof}
    Using the second moment method
\begin{equation}
\label{eq:bad_nodes}
\P(Z \ge 1)
\wge 1-\frac{\text{Var}(Z)}{(\E[Z])^2}
\weq 2 - \frac{\E[Z^2]}{(\E[Z])^2}.
\end{equation}
The Mecke equation from Theorem~\ref{thm:uni_mecke} along with the stationarity of the generated point process now yields
\begin{equation}
\E Z \weq \E \Big[\sum_{u \in V} \xi_{u}^V\Big] \weq \E \Big[\sum_{u \in V} \P_V(\cE_u^V)\Big] = \la n \P^0(\cE_0).
\label{eq:meanZ}
\end{equation}
The reader is referred to Appendix \ref{sec:ppp} for a brief discussion of theorems concerning Palm versions of Poisson point processes.

By writing $Z^2  = (\sum_{u\in V} \xi_u^V)^2 = \sum_u \xi_u^V + \sum_u \sum_{u'\ne u} \xi_u^V\xi_{u'}^V$, we find that
\begin{equation}
\E Z^2
\weq \E Z + \sum_u \sum_{u'\ne u} \P( \cE_u^V \cap \cE_{u'}^V )
\weq \E Z + \E \bigg[\sum_u \sum_{u'\ne u} \P_V( \cE_u^V \cap \cE_{u'}^V )\bigg].
\label{eq:Palm_second_moment}
\end{equation}
Therefore, using the bivariate Mecke equation from Theorem~\ref{thm:multi_mecke} and exploiting the stationarity of the generated point process, we get
\begin{align*}
\E Z^2
&\weq \E Z + \la^2 \int_{ (-\frac{n}{2}, \frac{n}{2}]} \int_{ (-\frac{n}{2}, \frac{n}{2}]}
\E \Big[ \P_{V\cup\{x,y\}}\big(\cE_x^{V\cup\{x,y\}} \cap \cE_y^{V\cup\{x,y\}}\big) \Big] dx \, dy \\
&\weq \E Z + \la^2 \int_{ (-\frac{n}{2}, \frac{n}{2}]} \int_{ (-\frac{n}{2}, \frac{n}{2}]}
\P^{0,y} \big( \cE_0^{0,y}\cap \cE_y^{0,y} \big) dx \, dy,
\end{align*}
where $\cE_u^{x,y} := \Big\{(V,\bsigma,\bA): (\bsigma,\bA_{u:})\in \cE_u^{V\cup \{x,y\}}\Big\}$. By letting $\xi^{0y}_u = \1_{\cE_u^{0,y}}$ for $u\in\{0,y\}$,  we have that
\begin{equation}
\frac{\E Z^2}{(\E Z)^2}
= \frac{1}{\E Z }
+ \frac{ \int_{ (-\frac{n}{2}, \frac{n}{2}]}
\E^{0,y} \Big[ \xi^{0y}_0 \xi^{0y}_y \Big] dy}{n \P^0(\cE_0)^2}
\label{eq:terms}
\end{equation}
From \eqref{eq:firstmomentcond} and \eqref{eq:meanZ}, the first term on the RHS in \eqref{eq:terms} tends to $0$ as $n\to \infty$, and the second term is at most equal to $1$ from \eqref{eq:secondmomentcond}. Therefore,  $\limsup_{n \to \infty} \frac{\E[Z^2]}{(\E[Z])^2} \le 1$ from \eqref{eq:terms}. Consequently, from \eqref{eq:bad_nodes}, there exists a bad node with high probability.
\end{proof}

In the following two subsections, we will show that if $\la I_\phi(p,q) <1$, then \eqref{eq:secondmomentcond} and \eqref{eq:firstmomentcond} hold.

\subsubsection{First moment analysis}
\label{sec:first_moment}
In this subsection, we show \eqref{eq:firstmomentcond}. For two distinct nodes $u,v\in V\cup\{0\}$, define
\begin{equation}
R_{uv} = \sigma_v\bigg[A_{uv}\log \Big(\frac{q}{p}\Big)+ \big( 1-A_{uv} \big)\log\bigg(\frac{1-qQ_{uv}}{1-pQ_{uv}}\bigg)\bigg].
\label{eq:ruv}
\end{equation}
To be concise, if $u$ is the origin we write $R_v \equiv R_{0v} $ and $\cL_u(k) = \cL_u(k,\bsigma_{\sim u},V,\bA)$ for $k\in \{-1,1\}$. Recall that $\cE_0 = \{(V,\bsigma,\bA): (\bsigma,\bA_{0:})\in \cE_0^V\}$.

\begin{prop}
\label{prop:origin_bad}
For all $\la >0$, $p,q \in(0,1)$, and geometric kernel $\phi$
with a bounded normalised interaction range $\zeronorm{\phi}$ satisfying $\la \zeronorm{\phi} >1$,  if $(V^{(n)}, \bsigma^{(n)}, \bA^{(n)}) \sim \GKBM_n(\la, \phi, p, q)$ and $\la I_{\phi}(p,q) <1$, then
\[
\lim_{n \to \infty}  n \P^0(\cE_0)
\weq \infty.
\]
\end{prop}
\begin{proof}
Conditioning on the community of the node at the origin,
\begin{equation}
\P^0(\cE_0)
= \sum_{k \in \{\pm 1\}} \frac12 \P^0\bigg(\frac{\cL_0(-k)}{\cL_0(k)} \ge 1\ \Big|\  \sigma_0 = k\bigg).
\label{eq:origin_bad}
\end{equation}
Consider the term with $\sigma_0 = +1$. Then
\[
\bigg\{\frac{\cL_0(-1)}{\cL_0(1)} \ge 1 \bigg\} \weq \bigg\{\sum_{v\in V\backslash \{0 \}} \log \frac{\Ber(P_{-1, \sigma_v}Q_{0v})(A_{0v})}{\Ber(P_{1,\sigma_v}Q_{0v})(A_{0v})} \ge 0\bigg\}
\]
Since the edges $A_{0v}$ are generated with the node at the origin being in the $+1$ community, for nodes with $\sigma_v=-1$, the log-likelihood ratio evaluates to
\begin{align*}
\log \frac{\Ber(P_{-1, \sigma_v}Q_{0v})(A_{0v})}{\Ber(P_{1,\sigma_v}Q_{0v})(A_{0v})} 
&= 
\log \bigg(\Big(\frac{pQ_{0v}}{qQ_{0v}}\Big)^{A_{0v}}\Big(\frac{1-pQ_{0v}}{1-qQ_{0v}}\Big)^{1-A_{0v}}\bigg)\\
&=
A_{0v}\log \Big(\frac{pQ_{0v}}{qQ_{0v}}\Big)+ \big( 1-A_{0v} \big)\log\bigg(\frac{1-pQ_{0v}}{1-qQ_{0v}}\bigg).
\end{align*}
A similar expression with a negative sign is obtained when $\sigma_v=+1$. Combining the two, we obtain
\begin{align}
\P^0\Big( \frac{\cL_0(-1)}{\cL_0(1)} \ge 1\Big|\sigma_0=+1\Big)& \weq \P^0\Big( \sum_{v\in \cV(0)}\sigma_v\Big[A_{0v}\log\frac{q}{p}+(1-A_{0v}) \log \frac{1-qQ_{0v}}{1-pQ_{0v}}\Big]\ge 0 \,\Big|\,\sigma_0=+1\Big) \nonumber\\
&\weq \sum_{m=0}^\infty \P(|\cV(0)|=m) \ \P^0\Big( \sum_{v=1}^m R_v \ge 0 \,\Big|\,\sigma_0=+1\Big)
\label{eq:map_Poisson_sum}
\end{align}
The same expression is obtained when $\sigma_0=-1$. Note that $Q_{0v} = \phi\Big(\frac{\|v\|}{\log n}\Big)$. In the following, we obtain a large deviation bound for the second probability term on the RHS. 
We proceed by first computing the moment generating function of $R_v$. To indicate the conditioning event $\sigma_0=+1$, we use the notation $\P^0_+, \E^0_+$ for the conditional (Palm) probability and expectation.

Let $\zeronorm{\phi} = \ka$. Note that given the number of points in the visible set of the origin, each node is distributed uniformly within $[-\ka \log n,\ka \log n]$, assigned community $\{+1,-1\}$ independently with equal probability, and an edge is drawn to the origin based on its community and location as in \eqref{eq:InteractionProduct}. Since the same procedure is performed independently for each of
the $m$ nodes, each of the $R_v$ variables has the same distribution. Moreover, $\big\{R_v, v=1,\cdots,m\big\}$ are all independent. Integrating out the community and location of node $v$, we have that
\[ 
\E^0_+\Big[\exp(t R_v)\Big] = \frac{1}{4\ka \log n} \int_{-\ka \log n}^{\ka \log n}  \Big[\E^0_+\big[\exp(t R_v)\big| \sigma_v=-1 ,v\big]+ \E^0_+\big[\exp(t R_v)\big| \sigma_v=+1,v\big]\Big]dv.
\]
Recall that $Q_{0v} = \phi\Big(\frac{\|v\|}{\log n}\Big)$. Since $\sigma_0 = +1$, the first expectation evaluates to 
\begin{align*}
\E^0_+\big[\exp(t R_v)\big| \sigma_v=-1,v\big]
&= \E^0_+\bigg[\exp\bigg(t\sigma_v\Big[A_{0v}\log\frac{q}{p}+(1-A_{0v}) \log \frac{1-qQ_{0v}}{1-pQ_{0v}}\Big] \bigg)\bigg| \sigma_v=-1,v \bigg]\\
&= \E^0_+\bigg[\exp\bigg(t\Big[A_{0v}\log\frac{p}{q}+(1-A_{0v}) \log \frac{1-pQ_{0v}}{1-qQ_{0v}}\Big] \bigg) \bigg]\\
&=(pQ_{0v})^t (qQ_{0v})^{1-t}+(1-pQ_{0v})^t(1-qQ_{0v})^{1-t}
\end{align*}
and similarly 
\[ \E^0_+\big[\exp(t R_v)\big| \sigma_v=+1,v\big] = (qQ_{0v})^t (pQ_{0v})^{1-t}+(1-qQ_{0v})^t(1-pQ_{0v})^{1-t}.\]
Therefore, we obtain
\begin{align}
\E^0_+\Big[\exp(t R_v)\Big]
&=\frac{1}{4\ka \log n} \int_{-\ka \log n}^{\ka \log n} \Big[(pQ_{0v})^t (qQ_{0v})^{1-t}+(1-pQ_{0v})^t(1-qQ_{0v})^{1-t} \nonumber\\ 
&\hspace{3.5cm}+(qQ_{0v})^t (pQ_{0v})^{1-t}+(1-qQ_{0v})^t(1-pQ_{0v})^{1-t}\Big] dv \nonumber\\
&=\frac{1}{2\ka \log n} \int_{0}^{\ka \log n} \Big[(pQ_{0v})^t (qQ_{0v})^{1-t}+(1-pQ_{0v})^t(1-qQ_{0v})^{1-t}\nonumber\\ 
&\hspace{3.5cm}+(qQ_{0v})^t (pQ_{0v})^{1-t}+(1-qQ_{0v})^t(1-pQ_{0v})^{1-t}\Big] d\|v\|
\label{eq:mgf1}
\end{align}
Putting $\frac{\|v\|}{\log n} = z$, we get
\begin{align}
\E^0_+\Big[\exp(t R_v)\Big]
&= \frac{1}{2\ka} \int_0^{\ka } \Big[(p\phi(z))^t (q\phi(z))^{1-t}+(1-p\phi(z))^t(1-q\phi(z))^{1-t} \nonumber\\
&\hspace{3.5cm}+(q\phi(z))^t (p\phi(z))^{1-t}+(1-q\phi(z))^t(1-p\phi(z))^{1-t}\Big] dz.
\label{eq:mgf2}
\end{align}
Since the above expression is symmetric with respect to $p,q$ and $t,1-t$, the integrand is symmetric around $t=\frac{1}{2}$. Thus, the moment generating function is symmetric around $\frac{1}{2}$ within $t\in[0,1]$. Since the moment generating function is convex in $t$, it is minimized when $t=\frac{1}{2}$. Therefore, the cumulant generating function defined as $\La(t) = \log \E^0_+\Big[\exp(t R_v)\Big]$ is minimized at $t=\frac{1}{2}$. The minimum value equals 
\begin{equation}
\La\Big(\frac{1}{2}\Big)= \log \Big(\frac{1}{\ka} \int_0^\ka \Big[\sqrt{pq}\phi(z)+\sqrt{(1-p\phi(z))(1-q\phi(z))}\Big] dz\Big).
\label{eq:min_mgf}
\end{equation}
From Cram\'{e}r's theorem, for $\alpha >\E ^0_+[R_v]$,
\[
\lim_{m \to \infty} \frac{1}{m} \log \P^0_+\Big(\sum_{v=1}^m R_v \ge \alpha m \Big)
= -\La^*(\alpha).
\]
where $\La^*(\alpha)$ is the Fenchel--Legendre transform of $\La(t)$ defined as 
$\La^*(\alpha) = \sup_{t \in \R}\big[t\alpha-\La(t)\big]$.
For $\alpha = 0$, this evaluates to $\La^*(0) = -\inf_{t\in\R} \La(t)$.
Note that, using a similar procedure as in \eqref{eq:mgf1} and \eqref{eq:mgf2},
the expected value of $R_v$ can be evaluated to be
\begin{align*}
 \E^0_+[R_v] 
 \weq \frac{1}{4\ka \log n} \int_{-\ka \log n}^{\ka \log n}
   \bigg[pQ_{0v} \log \frac{qQ_{0v}}{pQ_{0v}}+(1-pQ_{0v})\log \frac{1-qQ_{0v}}{1-pQ_{0v}}\bigg]dv
 \ < \ 0,
\end{align*}
since the integrand is the negative of the KL divergence between two Bernoulli distributions with parameter $pQ_{0v}$ and $qQ_{0v}$.
Thus Cram\'{e}r's theorem is applicable with $\alpha = 0$.
Since $\La(\cdot)$ is convex, the infimum of the cumulant generating function $\La(t)$ 
is achieved at $t=\frac{1}{2}$ and we obtain for any $\ga>0$ and a large enough $m$
\[
\bigg| \frac{1}{m} \log\P^0_+\Big(\sum_{v=1}^m R_v \ge 0 \Big) - \La \Big(\frac{1}{2}\Big)\bigg|
\wle \ga.
\]
A similar large deviation bound is obtained with $\P_+^0(\cdot)$ replaced by $\P_-^0(\cdot)$. 
Using the above equation in \eqref{eq:map_Poisson_sum}, for any $\ga>0$ there exists an $m_0$ large enough such that ,
\[
 \P^0\Big( \frac{\cL_0(-1)}{\cL_0(1)}
 \wge 1\Big|\sigma_0=+1\Big)
 \wge \sum_{m=m_0}^\infty \P(|\cV(0)|=m) \exp\Big(m\Big(\La\Big(\frac{1}{2}\Big)-\ga\Big)\Big). 
\]
Including the initial terms of the summation, we obtain
\begin{align*}
 \P^0\Big( \frac{\cL_0(-1)}{\cL_0(1)} 
 &\wge 1 \Big|\sigma_0=+1\Big)\\
 &\wge  \sum_{m=0}^\infty \P(|\cV(0)|=m)\exp\Big(m\Big(\La\Big(\frac{1}{2}\Big)-\ga\Big)\Big) - \sum_{m=0}^{m_0} \P(|\cV(0)|=m)\exp\Big(m\Big(\La\Big(\frac{1}{2}\Big)-\ga\Big)\Big)\\
 &\wge \sum_{m=0}^\infty \P(|\cV(0)|=m)
  \exp\Big(m\Big(\La\Big(\frac{1}{2}\Big)-\ga\Big)\Big) - \P(|\cV(0)| \le m_0).
\end{align*}
Since $|\cV(0)|$ is a Poisson random variable with mean $2 \la \ka \log n$,
the first term is its moment generating function evaluated at
$\La(\frac{1}{2})-\ga$.
For a random variable $X\sim \Poi(\mu)$,
$\E[e^{tX}] = \exp\big(\mu (e^t-1)\big)$. 
This yields
\begin{align*}
 \P^0\Big( \frac{\cL_0(-1)}{\cL_0(1)} \ge 1\Big|\sigma_0=+1\Big)
 &\wge \exp\Big[2 \la \ka \log n \big(e^{\La(\frac{1}{2})-\ga}-1\big) \Big]-\P(|\cV(0)| \le m_0)\\ 
 &\weq n^{2 \la \ka \big(e^{\La(\frac{1}{2})-\ga}-1\big)} -\P(|\cV(0)| \le m_0).
\end{align*}
A similar computation results in 
\[
 \P^0\Big( \frac{\cL_0(+1)}{\cL_0(-1)} \ge 1\Big|\sigma_0=-1\Big)
 \wge n^{2 \la \ka \big(e^{\La(\frac{1}{2})-\ga}-1\big)} -\P(|\cV(0)| \le m_0),
\]
which together when substituted in \eqref{eq:origin_bad} yields 
\begin{equation}
 \label{eq:e0withm0}
 \P^0(\cE_0 )
 \wge n^{2 \la \ka \big(e^{\La(\frac{1}{2})-\ga}-1\big)} -\P(|\cV(0)| \le m_0),
\end{equation}
Since $e^{\La(\frac{1}{2})} =1-\frac{I_{\phi}(p,q)}{2\ka}$ from \eqref{eq:min_mgf}, and 
\[
\lim_{\ga \to 0} 2 \la \ka \big(e^{\La(\frac{1}{2})-\ga}-1\big) \weq 2 \la \ka \big(e^{\La(\frac{1}{2})}-1\big) \weq -\la I_{\phi}(p,q)>-1,
\]
taking $\beta = \frac{1+2 \la \ka \big(e^{\La(\frac{1}{2})}-1\big)}{3} > 0$, we can choose a $\ga$ small enough such that 
\begin{equation}
2 \la \ka \big(e^{\La(\frac{1}{2})-\ga}-1\big) > -1+2\beta.
\label{eq:beta_term}
\end{equation}
Using \eqref{eq:beta_term} in \eqref{eq:e0withm0}, we obtain 
\begin{equation}
\P^0(\cE_0 )
\wge n^{-1+2\beta} -\P(|\cV(0)| \le m_0).
\label{eq:e0withbetam0}
\end{equation}
The latter term in \eqref{eq:e0withbetam0} is the tail probability of a Poisson random variable whose mean is $2\la \ka \log n$. Let $c<2\beta$ be a constant and $m_0 = \ga' \log n$. We will show that for an appropriate choice of $\ga'$, $\P(|\cV(0)| \le \ga' \log n) \le n^{-1+c}$. Indeed,
using a standard Chernoff bound (Lemma~\ref{lem:conc_poisson_new}), we obtain
\[
\P(|\cV(0)| \le m_0) 
\weq \P(|\cV(0)| \le \ga' \log n)
\wle n^{-2 \la \ka h\big(\frac{\ga'}{2 \la \ka}\big)},
\]
where
$h(x) = x \log x + 1 - x$.
Note that $\lim_{\ga'\to 0}h\big(\frac{\ga'}{2 \la \ka}\big) = 1$, and $h\big(\frac{\ga'}{2 \la \ka}\big)$ is strictly decreasing for $0<\ga'<2\la \ka$. Since $2\la \ka \ge 1$, for sufficiently small $\ga'$, $2 \la \ka h\big(\frac{\ga'}{2 \la \ka}\big) >1-c$ and we obtain $\P(|\cV(0)| \le m_0) \le n^{-1+c}$. Substituting in \eqref{eq:e0withbetam0}, since $c<2\beta$, we can write
\[
\P^0(\cE_0 )
\wge n^{-1+\beta},
\]
where $\beta>0$ whenever $\la I_{\phi}(p,q)<1$.
\end{proof}

\subsubsection{Second moment analysis}
\label{sec:second_moment}
\begin{prop}
For all $\la >0$, $p, q \in (0,1)$, and geometric kernels $\phi$ with a bounded normalised interaction range $\zeronorm{\phi}$, if $\la I_\phi(p,q) < 1$, then the graph $G_n \sim \GKBM_n(\la, \phi, p, q)$ satisfies condition \eqref{eq:secondmomentcond}.
\label{prop:second_momentz}
\end{prop}

\begin{proof}
With $\ka = \zeronorm{\phi}$ as defined in \eqref{eq:kappa}, we have
\begin{align*}
\int_{(-\frac{n}{2}, \frac{n}{2}]}
\E^{0,y} \Big[ \xi^{0y}_0 \xi^{0y}_y \Big] dy
&\weq \int_{B\left(0,2\ka\log n\right)} \E^{0,y} \Big[ \xi^{0y}_0 \xi^{0y}_y \Big] dy
+ \int_{(-\frac{n}{2}, \frac{n}{2}] \cap B \left(0,2\ka\log n\right)^{\text{c}}}
\E^{0,y} \Big[ \xi^{0y}_0 \xi^{0y}_y \Big] dy\\
&\wle \int_{B\left(0,2\ka\log n\right)}
\E^{0,y} \Big[\xi^{0y}_0 \Big] dy
+ \int_{(-\frac{n}{2}, \frac{n}{2}] \cap B\left(0,2\ka\log n\right)^{\text{c}}}
\E^{0,y} \Big[ \xi^{0y}_0 \xi^{0y}_y \Big] dy.
\end{align*}
Owing to spatial independence of the Poisson point process and our choice of $\ka$, for two nodes at $x$ and $y$ that are at least a distance of $d(x,y) > 2\left(\ka\log n\right)$ apart, we have 
\begin{align*}
\E^{x,y} \Big[\xi^{xy}_x \xi^{xy}_y\Big]
\weq \E^{x} \Big[\xi^{V\cup \{x\}}_x\Big] \, \E^{y} \Big[\xi^{V\cup \{y\}}_y\Big]
\weq \Big(\E^{x} \Big[\xi^{V\cup \{x\}}_x\Big]\Big)^2.
\end{align*}
where the last equality is due to translation invariance on the torus. Using $\xi_0^0 = \xi_0^{V\cup\{0\}}$, we obtain
\begin{align*}
\frac{\int_{y \in (-\frac{n}{2}, \frac{n}{2}]} \E^{0,y} \Big[ \xi^{0y}_0 \xi^{0y}_y \Big] dy}
{n \left( \P^0(\cE_0) \right)^2}
&\wle \frac{4 \ka \log n \E^{0} \Big[ \xi_0^0 \Big]
+ \left(n - 4 \ka \log n \right) ( \E^{0} \xi_0^0 )^2}{n \left( \P^0(\cE_0) \right)^2} \\
&\weq \left(\frac{4 \ka \log n}{n} \right) \frac{1}{\P^0(\cE_0)}+\left(1-\frac{4\ka \log n}{n}\right).
\end{align*}
Using Proposition \ref{prop:origin_bad}, for $\la I_{\phi}(p,q) < 1$
\begin{equation}
n\P^0( \cE_0) = n^{1-\la I_{\phi}(p,q)} \rightarrow \infty 
\end{equation}
as $n \rightarrow \infty$. This gives the desired result in the statement of the proposition.
\end{proof}

\subsection{Proof of Theorem~\ref{thm:impossibility}}
\label{sec:proof_impossibility}
In this subsection, we tie up the results from this section to prove Theorem~\ref{thm:impossibility}.
\begin{proof}[Proof of Theorem~\ref{thm:impossibility}]
It was already shown in Lemma~\ref{lem:connectivity_softrgg} that if $\la \zeronorm{\phi} < 1$, the graph $G'$ is disconnected. 
Any algorithm for recovering node communities in $G$ can do so only if there is a single connected component in $G'$. When there are multiple components, the algorithm recovers a community assignment for each component. However, one can obtain another valid community assignment by flipping the node communities in one component while retaining the assignments in other components. This is possible since there are no interactions (neighbours or non-neighbours) across components. However, only one of these community assignments corresponds to the ground truth up to a global flip. Thus,
it is impossible for any algorithm to unanimously decide the node communities. In other words, exact recovery is not possible. This proves the necessity of condition (a) in the statement of the theorem.

For the condition $\la I_{\phi}(p,q) < 1$, note that the statements of Propositions \ref{prop:origin_bad} and \ref{prop:second_momentz} imply  \eqref{eq:secondmomentcond} and \eqref{eq:firstmomentcond}. From Lemma~\ref{lem:first_second_moment_cond} there exists a bad node with high probability i.e., $\lim_{n\to \infty}\P(Z \ge 1) \weq 1$ whenever $\la I_\phi(p,q) <1$. Using Lemma~\ref{lem:compMAPnec}, presence of a bad node indicates that the MAP estimate is not unique, or not equal to the ground-truth up to a global sign flip. Therefore, under the same conditions, community structure cannot be recovered exactly.
\end{proof}

\section{Analysis of Algorithm~\ref{alg:full}}
\label{sec:achievability}

In this section we prove Theorem~\ref{thm:exact_recovery} by 
carrying out a detailed analysis of Algorithm~\ref{alg:full}.
As a preliminary step we also prove (Theorem~\ref{thm:almost_exact}) that
the initialization and the propagation phases in Algorithm~\ref{alg:full} recover the community memberships almost exactly.

We consider a realisation $(V, \{\sigma_v\}, \{A_{uv}\})$ sampled from the $\GKBM_n(\la, \phi, p, q)$ model in which
$\la \zeronorm{\phi} > 1$ and $\la I_\phi(p,q) > 1$, and 
\[
 \eps = \inf_{x \le \zeronorm{\phi}} \phi(x)
\] 
satisfies $\eps > 0$. We analyse Algorithm~\ref{alg:full} with input $(V, \{A_{uv}\})$,
where the resolution parameter $\chi > 0$
and the threshold parameter $\de > 0$ are chosen small enough according to \eqref{eq:chi} and \eqref{eq:delta}.
We denote the segments of length $\chi\log n$ that partition the circle by
\[
 C_i, \quad i=1,\cdots,\frac{n}{\chi \log n}.
\]
The set of nodes contained in a segment $C$ is denoted by $V(C) = V \cap C$, and
the segment is called \emph{$\de$-occupied} if $|V(C)| \ge \de \log n$.
The $\de$-occupied segments of the partition $\{C_i\}$ are denoted by
\[
 B_j, \quad j = 1,\cdots, J.
\]
Furthermore, we often abbreviate $\ka = \zeronorm{\phi}$ and $V_j = V(B_j)$.

\subsection{Preliminaries}

In this subsection, we obtain a few results that will be required for the analysis of the algorithm.

\subsubsection{Number of nodes in each segment}

\begin{lem}
\label{lem:max_nodes}
Let $V$ be a homogeneous Poisson point pattern with intensity $\la > 0$
on a circle of circumference $n$ that is partitioned into segments $\{C_i\}$
of length $\chi \log n$.
Then
\[
 \P \Big( \max_i \abs{V(C_i)} \le \De \log n \Big)
 \wge 1 - \frac{1}{\chi \log n},
\]
where $\De \ge \la\chi+1 + \sqrt{2\la\chi+1}$.
\end{lem}
\begin{proof}
The number of nodes $|V(C_i)|$ in a segment $C_i$ of length $\chi \log n$
is Poisson-distributed with mean $\la \chi \log n$.
Using the Chernoff bound from Lemma~\ref{lem:conc_poisson}, we obtain
\begin{align*}
 \P(|V(C_i)| > \Delta \log n )
 &\wle \exp \left(-\frac{(\Delta-\la \chi)^2\log n}{2\Delta}\right)
 \weq n^{- \frac{(\Delta-\la \chi)^2  }{2\Delta}}.
\end{align*}
Our choice of $\De$ implies that $\frac{(\Delta-\la \chi)^2  }{2\Delta} \ge 1$.
The union bound now gives
\[
 \P \Big( \max_i \abs{V(C_i)} > \De \log n \Big)
 \wle \frac{n}{\chi \log n} n^{- \frac{(\Delta-\la \chi)^2  }{2\Delta}}
 \wle \frac{1}{\chi \log n},
\]
so the claim follows.
\end{proof}

Similarly, we will also need the following lemma to bound the number of nodes in a segment from below.

\begin{lem}
\label{lem:single_block_nodes}
Suppose $C$ is a segment of length $\nu \log n$ with $\la \nu >1$.
Then for any $0 < \al < \la \nu - 1$,
\[
 \P(|V(C)| > \be \log n)
 \wge 1-n^{-1-\al}
\]
with $\be = \la\nu h^{-1} ( \frac{1+\al}{\la\nu} )$,
where $h^{-1}(\cdot)$ denotes the inverse of $h(x) = x \log x + 1 - x$ on $(0,1)$.
\end{lem}
\begin{proof}
The number of nodes within $C$ is a Poisson random variable with mean $\la \nu \log n$,
so the result follows directly from Lemma~\ref{lem:conc_poisson_new}.
\end{proof}

\subsubsection{Presence of a connected skeleton}
\label{sec:skeleton}

Line~2 of Algorithm~\ref{alg:full} chooses the sequence of $\de$-occupied segments
$\{B_j,j=1,\cdots,J\}$ for the propagation step.
We refer to this sequence as the \emph{$\de$-skeleton}.
The $\de$-skeleton is called \emph{$(\ka,\chi)$-connected} if between any $\de$-occupied segments
$B_j$ and $B_{j+1}$, there are at most $\lfloor\frac{\ka}{\chi}\rfloor - 2$ segments
that are not $\de$-occupied.
This requirement implies that all nodes in $B_j \cup B_{j+1}$ are within distance $\ka \log n$ of each other.
This is crucial for propagating labels from one $\delta$-occupied segment to another in the Propagation phase.
The following lemma provides a sufficient condition for the $\de$-skeleton to be $(\ka,\chi)$-connected.

\begin{lem}
\label{lem:connected_vis_graph}
Assume that $\la$ and $\ka = \zeronorm{\phi}$
satisfy $\la \ka > 1$, and that the parameters 
$\chi,\de > 0$ satisfy \eqref{eq:chi}--\eqref{eq:delta}.
Let $\cH$ be the event that the
$\de$-skeleton $\{B_j\}$ is $(\ka,\chi)$-connected.
Then there exists a number $n_0 > 0$ such that
\[
 \P(\cH^c) \wle \frac{1}{\chi \log n}
 \qquad \text{for all $n \ge n_0$}.
\]
\end{lem}

\begin{proof}
Let $r= \lfloor\frac{\ka}{\chi}\rfloor-1$ and $\{C_i:i=1,\dots,\lceil\frac{n}{\chi \log n}\rceil\}$
be all the segments numbered in the clockwise direction.
The $\de$-skeleton is not $(\ka,\chi)$-connected if and only if there exist $r$ consecutive segments each containing less than $\de \log n$ points.  Then, using indices modulo $\big\lceil\frac{n}{\chi \log n}\big\rceil$,
\begin{align*}
 \P(\cH^c)
 &\wle \sum_{i=1}^{\lceil\frac{n}{\chi \log n}\rceil}
 \P\big( |V(C_{i+1})| < \de \log n, \dots, |V(C_{i+r})| < \de \log n \big) 
\end{align*}
Let 
\(
 U_i := \bigcup_{m=\LLupdated{1}}^r C_{i+m}.
\) 
If each of the segments $C_{i+1},\dots,C_{i+r}$ has at most $\de \log n$ nodes, then $|V(U_i)| \le \frac{\ka}{\chi}\de \log n$ since $r \le \frac{\ka}{\chi}$.
Hence
\begin{equation*}
 \P\big( |V(C_{i+1})| < \de \log n, \dots, |V(C_{i+r})| < \de \log n \big) 
 \wle \P\Big(|V(U_i)| \le \frac{\ka}{\chi} \de \log n \Big).
\end{equation*}
Note that $\volume(U_i) = \big(\lfloor \frac{\ka}{\chi}\rfloor-1\big)\chi \log n \ge \nu \log n$,
where $\nu = \ka - 2\chi$. Note also that \eqref{eq:delta} implies that $\frac{\ka}{\chi}\de \le \be$,
with $\be = \la \nu h^{-1}( \frac12 + \frac{1}{2\la\nu})$.
By applying Lemma~\ref{lem:single_block_nodes} with 
and
$\al = \frac12(\la\nu-1)$, and noting that $\be = \la \nu h^{-1}( \frac{1+\al}{\la\nu})$,
it follows that 
\[
 \P\Big(|V(U_i)| \le \frac{\ka}{\chi} \de \log n \Big)
 \wle n^{-1-\al}.
\]
Hence
\begin{align*}
 \P(\cH^c)
 \wle \bigg(\frac{n}{\chi \log n} + 1 \bigg) n^{-1-\al}
 \weq \frac{1}{\chi n^{\al} \log n} + \frac{1}{n^{1+\al}}.
\end{align*}
We conclude that
$\P(\cH^c) \le \frac{1}{\chi \log n}$ for all $n$ large enough
so that
$n^\al \ge 2$ and $n^{1+\al} \ge 2 \chi \log n$.
\end{proof}

\subsubsection{Additional definitions}

In this section, we introduce a few definitions which will be required in the analysis of Algorithm~\ref{alg:full}.
Line~\ref{line:choose_u0}  chooses an initial node $u_0 \in V_1$ and assigns $\hsigma_{u_0} = +1$. The node communities are obtained relative to that of node $u_0$. This means that if $\sigma_{u_0} = -1$, the recovered node communities are the negation of the ground-truth communities. To formalize this notion, we make the following definition.
\begin{defn}
For $S \subseteq \torus$ (either a segment or a set),
the restricted Hamming distance between two community membership vectors $\tbsigma$ and $\bsigma$ relative to $u_0 \in V$ is defined as 
\[
 \Ham_{S}(\tbsigma,\bsigma)
 = |\{v\in V(S) \colon \tsigma_v \ne \sigma_{u_0} \sigma_v\}|.
\]
\end{defn}
\begin{rem}
Note that for any estimate $\hbsigma$, $\Ham(\hbsigma,\bsigma) \le \Ham_V(\hbsigma,\bsigma)$ since $\sigma_{u_0} \in \{-1,+1\}$. Therefore, it suffices to show almost-exact and exact recovery with respect to the Hamming distance relative to $u_0$.
\label{rem:Hamming_dist_relation}
\end{rem}

For discrete probability measures $P,Q$,
the \Renyi divergence of order $\al \ne 1$
is denoted by
$D_\al(P\|Q) = (\alpha - 1)^{-1} \log \sum_x P(x)^\al Q(x)^{1-\al}$,
and the Hellinger distance is defined by
$\Hel^2(P,Q) = \frac{1}{2}\sum_x (\sqrt{P(x)}-\sqrt{Q(x)})^2$.
In particular,
\begin{equation}
 D_{1/2}(P\|Q) = -2\log \sum_x \sqrt{P(x)Q(x)}
 \qquad \text{and} \qquad 
 D_{3/2}(P\|Q)
 = 2\log \sum_x \frac{P(x)^{3/2}}{Q(x)^{1/2}}.
 \label{eq:renyi_div_eval}
\end{equation}
We write $D_{1/2}(P,Q) = D_{1/2}(P \| Q)$ to highlight that
$D_{1/2}$ is symmetric in its arguments. 
We also note that
$D_{1/2}(P,Q) \ge \Hel^2(P,Q)$ and $D_\alpha(P\|Q)$ is monotonically increasing in $\alpha$ (see \cite{van2014renyi}).

The rest of this section analyses the Initialization, Propagation and the Refinement phases, and culminates by proving Theorem~\ref{thm:exact_recovery}.

\subsection{Initialization phase}
Line~\ref{line:shorthand} of Algorithm~\ref{alg:full} introduces a shorthand notation $V_j$ for the set of nodes present in the $\delta$-occupied segment $B_j$. Given the locations and community labels of nodes $u_0,u,$ and $v$, the random variable $A_{uv}A_{u_0v}$ is distributed as 
\begin{equation}
A_{uv}A_{u_0v} \sim \begin{cases}
\Bern\left(Q_{uv}Q_{u_0v}p^2\right), & \text{if } \sigma_u = \sigma_{u_0}=\sigma_v, \\
\Bern\left(Q_{uv}Q_{u_0v}q^2\right), & \text{if } \sigma_u = \sigma_{u_0}\ne\sigma_v, \\
\Bern\left(Q_{uv}Q_{u_0v}pq\right),  & \text{if } \sigma_u \ne \sigma_{u_0}.
\end{cases}
\label{eq:dist_comm_nbrs}
\end{equation}

Line~\ref{line:count_common_nbrs} of Algorithm~\ref{alg:full} computes the number of common neighbours of $u_0$ and $u$ within $B_1$. This is compared with the average number of common neighbours $M(u,u_0,B_1)$ in Line~\ref{line:init_compare}. Note that $M(u,u_0,B_1) =\Theta(\log n)$ since
\begin{equation}
\eps^2 \delta \log n 
\wle 
\sum_{v\in V_1\setminus \{u,u_0\}} Q_{uv}Q_{u_0v}
\wle  \Delta \log n.
\label{eq:M_theta_logn}
\end{equation}

Define the events $\cT_{u_0,u} := \{N_{u_0,u}< M(u,u_0,B_1)\}$ and $\cI_1 := \{|V_1| \in [\delta \log n,\Delta \log n]\}$. Let $\P_{V_1}$ be the probability distribution conditioned on the nodes within $V_1$. The following two propositions bound the probability that Lines 5--9 of Algorithm~\ref{alg:full} make an error in recovering the community of node $u$ depending on whether $u$ and $u_0$ are in the same or different community. 

\begin{prop}
There exist constants $c_1,c_2>0$ such that
for any $u\in V_1 \setminus \{u_0\}$:
\begin{itemize}
\item[(a)] If $u$ and $u_0$ are in different communities, then $\P_{V_1}(\cT_{u_0,u}^c|\sigma_u\ne\sigma_{u_0},\cI_1) \le n^{-c_1\delta \eps^4}$
\item[(b)] If $u$ and $u_0$ are in the same community, then $\P_{V_1}(\cT_{u_0,u}|\sigma_u = \sigma_{u_0},\cI_1) \le n^{-c_2 \delta \eps^4}$
\end{itemize}
\label{prop:type2}
\end{prop}

\begin{proof}
\emph{Part (a):} Given the communities and locations of nodes within $V_1$, the number of common neighbours $N_{u_0,u}$ is a sum of independent Bernoulli random variables with mean $pq\sum_{v\in V_1} Q_{uv}Q_{u_0v}$ when $\sigma_u\neq\sigma_{u_0}$. Conditioning on the community assignment within $V_1$ and using Hoeffding's inequality (see Lemma~\ref{lem:hoeffding}), we obtain
\begin{align*}
    \P_{V_1}(N_{u_0,u}&> M(u,u_0,B_1)\big| \sigma_u\ne\sigma_{u_0},\cI_1) \\
    &= \frac{1}{2^{|V_1|-2}} \sum_{\bsigma_{V_1\setminus \{u,u_0\}}} \P_{V_1}\bigg(\sum_{v\in V_1\setminus \{u,u_0\}} A_{uv}A_{u_0v} > M(u,u_0,B_1)\Big|\sigma_u\ne\sigma_{u_0},\cI_1,\bsigma_{V_1\setminus \{u,u_0\}}\bigg)\\
    &\le \frac{1}{2^{|V_1|-2}} \sum_{\bsigma_{V_1\setminus \{u,u_0\}}} \exp\bigg[\frac{-2(M(u,u_0,B_1)-pq\sum_{v\in V_1\setminus \{u,u_0\}} Q_{uv}Q_{u_0v})^2}{|V_1|-2}\bigg]\\
    &\le \exp\bigg[-\frac{2\big(\frac{(p+q)^2}{4}-pq\big)^2}{|V_1|-2} \big(\eps^2(|V_1|-2)\big)^2\bigg]\\
    &\le \exp\bigg[-2 \eps^4 \delta \log n \bigg(\frac{(p+q)^2}{4}-pq\bigg)^2  \bigg].
\end{align*}
Therefore, $\P_{V_1}(N_{u_0,u}> M(u,u_0,B_1)\big|\sigma_u\ne\sigma_{u_0},\cI_1,) \le n^{-c_1 \delta \eps^4}$ where $c_1 =  \big(\frac{(p+q)^2}{4}-pq\big)^2 $.

\emph{Part (b):}
We proceed on similar lines as in the proof of part (a). The expected value of $N_{u_0,u}$ can be computed as 
\begin{align*}
\E_{V_1}[N_{u_0,u} | \sigma_u=\sigma_{u_0},\cI_1] 
&= \sum_{v \in V_1\setminus \{u,u_0\}}\E [A_{uv}A_{u_0v}| \sigma_u=\sigma_{u_0},\cI_1] \\
&= \sum_{v \in V_1\setminus \{u,u_0\}}\frac{1}{2} \E [A_{uv}A_{u_0v}| \sigma_u=\sigma_{u_0}=\sigma_v,\cI_1] +\frac{1}{2} \E [A_{uv}A_{u_0v}| \sigma_u=\sigma_{u_0}\neq \sigma_v,\cI_1]\\
&= \sum_{v \in V_1\setminus \{u,u_0\}} \frac{p^2+q^2}{2}Q_{uv}Q_{u_0v}.
\end{align*}
Since $N_{u_0,u}$ is a sum of independent Bernoulli random variables, using Hoeffding's inequality again, we obtain
\begin{align*}
\P_{V_1}(N_{u_0,u} < M(u,u_0,B_1)\big|& \sigma_u\ne\sigma_{u_0},\cI_1) \\
&\le \frac{1}{2^{|V_1|-2}} \sum_{\bsigma_{V_1\setminus \{u,u_0\}}} \exp\bigg[\frac{-2(M(u,u_0,B_1)-\frac{p^2+q^2}{2}\sum_{v\in V_1\setminus \{u,u_0\}} Q_{uv}Q_{u_0v})^2}{|V_1|-2}\bigg]\\
&\le \exp\bigg[-2\frac{\big(\frac{p^2+q^2}{2}-\frac{(p+q)^2}{4}\big)^2}{|V_1|-2} \big(\eps^2(|V_1|-2)\big)^2\bigg]\\
&\le \exp\bigg[-\eps^4\delta \log n \bigg(\frac{p^2+q^2}{2}-\frac{(p+q)^2}{4}\bigg)^2  \bigg]
\end{align*}
which proves the second part of the proposition with $c_2= \big(\frac{p^2+q^2}{2}-\frac{(p+q)^2}{4}\big)^2 $.
\end{proof}

The following lemma is the main result of the Initialization phase and asserts that Lines 3--9 of Algorithm~\ref{alg:full} recover the communities of all nodes within block $B_1$ with high probability. 

\begin{lem}
\label{lem:error_all_nodes_phase1}
The Initialization phase of Algorithm~\ref{alg:full} recovers the communities of nodes in the initial block $B_1$ with high probability, i.e., there exists $c>0$ such that
\[
\P_{V_1} \left(
\Ham_{B_1}(\tbsigma, \bsigma) = 0 \ \Big| \ \cI_1
\right)
\wge 1 - \frac{\Delta \log n}{ n^{c \delta \eps^4 }}.
\]
\end{lem}
\begin{proof} 
As a consequence of Proposition \ref{prop:type2}, the probability of making an error in estimation of the community of any node $u\in V_1\setminus \{u_0\}$ can be bounded as 
\begin{align*}
\P_{V_1}(\tsigma_u \ne \sigma_{u_0}\sigma_u|\cI_1) &=  \P_{V_1}(\tsigma_u \ne \sigma_{u_0} \sigma_u\,|\,\sigma_u\ne\sigma_{u_0},\cI_1)\,\P_{V_1}(\sigma_u\ne\sigma_{u_0})\\
&\hspace{2cm}+ \P_{V_1}(\tsigma_u \ne \sigma_{u_0} \sigma_u\,|\,\sigma_u=\sigma_{u_0},\cI_1)\,\P_{V_1}(\sigma_u=\sigma_{u_0})\\ 
&= \P_{V_1}(\cT_{u_0,u}^c|\sigma_u\ne\sigma_{u_0},\cI_1)\ \frac{1}{2}+ \P_{V_1}(\cT_{u_0,u}|\sigma_u=\sigma_{u_0},\cI_1)\ \frac{1}{2}\\
&\le n^{-c \delta \eps^4},
\end{align*}
where $c=\min \{c_1,c_2\}$ and the constants $c_1,c_2$ are the ones in Proposition \ref{prop:type2}. Using the union bound, we obtain
\begin{align*}
\P_{V_1}\left(\bigcap_{u \in V_1}\{\tsigma_u = \sigma_{u_0} \sigma_u\}\Big| \cI_1 \right)  
&\ge 1-\sum_{u\in V_1}\P_{V_1}\left(\tsigma_u \ne \sigma_{u_0} \sigma_u\Big| \cI_1 \right) \\
&\ge 1- |V_1| n^{-c\delta \eps^4}\\
&\ge 1- \frac{\Delta\log n}{n^{c\delta \eps^4}}.
\end{align*}
\end{proof}

\subsection{Propagation phase}
\label{sec:propa}
Lines 10--15 of Algorithm~\ref{alg:full} constitute the propagation phase in which the communities recovered in the initial segment are propagated to successive $\delta$-occupied segments as shown in Fig.~\ref{fig:block_division}. The analysis of the propagation phase is done in three steps as described below.
\begin{itemize}
\item \textbf{Step 1:} We first obtain the probability of making an error in assigning the community to a node $u\in V_{j+1}$ given the estimated communities of all nodes in $V_j$. This allows us to evaluate the number of mistakes made in segment $B_{j+1}$ given the node communities in $B_j$.
\item \textbf{Step 2:} Using a coupling argument we show that the number of mistakes in segment $B_{j+1}$ is at most a constant, $M$, given the communities and the number of mistakes in segment $B_j$ with overwhelming probability.
\item \textbf{Step 3:} Propagating over successive segments incurs a small drop in probability for there being $M$ errors in the next segment. This drop in probability can be made small thus recovering the communities of nodes in all $\delta$-occupied segments. The estimator thus obtained recovers the communities almost exactly. 
\end{itemize}

In our analysis, we make use of the following constants. Let
\begin{align}
\xi_1(p,q,\eps) 
&= 
\max\Bigg\{2\log \left[\frac{p^{3/2}}{q^{1/2}}+\frac{(1-p\eps)^{3/2}}{(1-q)^{1/2}}\right], 
2\log \left[\frac{q^{3/2}}{p^{1/2}}+\frac{(1-q\eps)^{3/2}}{(1-p)^{1/2}}\right]\Bigg\}, 
\quad 
\text{ and } \nonumber\\
\xi_2(p,q,\eps) 
&=
\eps (\sqrt{p}-\sqrt{q})^2.
\label{eq:xi1xi2}
\end{align}
Further, let 
\begin{equation}
M 
=\frac{10}{4\delta \eps(\sqrt{p}-\sqrt{q})^2}, \quad \eta_1 = e^{\xi_1 M}, \quad \text{ and } \quad c'=\frac{\delta \xi_2}{2}.
\label{eq:M}
\end{equation}

\subsubsection{Step 1: Propagation error for a single node}
In this subsection, we evaluate the probability of making an error in estimating a node's community in the subsequent occupied segment during the propagation phase. Before we proceed, we introduce a few definitions and notations which will be useful in the following analysis. For a $\delta$-occupied segment with nodes $V_j$, define
\begin{align}
\cZ_{++}(V_j) &= \{v\in V_j : \sigma_v = \sigma_{u_0}, \tsigma_v=+1\}, \nonumber\\
\cZ_{+-}(V_j) &= \{v\in V_j : \sigma_v = \sigma_{u_0}, \tsigma_v=-1\}, \nonumber\\
\cZ_{-+}(V_j) &= \{v\in V_j : \sigma_v \neq \sigma_{u_0}, \tsigma_v=+1\}, \nonumber\\
\cZ_{--}(V_j) &= \{v\in V_j : \sigma_v \neq \sigma_{u_0}, \tsigma_v=-1\}.
\label{eq:error_nbrs}
\end{align}
To describe in words, $\cZ_{+-}(V_j)$, for example, is the set of nodes  $v\in V_j$ that belong to the ground-truth community $\sigma_v = \sigma_{u_0}$ and get assigned a label $\tsigma_v = -1$ in the propagation phase. Naturally, $\cZ_{+-}(V_j) \cup \cZ_{-+}(V_j)$ constitute all the mistakes that the Propagation phase makes in $V_j$ for $j\ge2$. 
Proposition \ref{prop:prop_single_error} below evaluates the probability of making an error in assigning the community of a node in Line \ref{line:prop_assignment} of Algorithm~\ref{alg:full}. 
\begin{prop}
\label{prop:prop_single_error}
Consider the $\delta$-occupied segments $B_j$ and $B_{j+1}$ for any $1\le j\le J-1$. Given that there are at most $M$ mistakes in segment $B_j$, the probability of making an error in assigning node $u\in V_{j+1}$ to its community by the propagation phase is bounded as
\[
 \P_V \Big(\tsigma_u \ne \sigma_{u_0}\sigma_u \, \Big| \,
 \tilde{\sigma}(V_j), \, \sigma(V_j), \, |\cZ_{+-}(V_j)| + |\cZ_{-+}(V_j)| \le M \Big)
 \wle 
 \eta_1n^{-c'},
\]
where $\eta_1, c' $ are defined in \eqref{eq:M}.
\end{prop}

\begin{proof}
We begin by evaluating the probability $\P_V(\tsigma_u \ne \sigma_{u_0}\sigma_u \,|\,\tilde{\sigma}(V_j),\sigma(V_j))$. Due to the symmetry in assigning node labels, it suffices to evaluate $\P_V(\tsigma_u= -1 \,|\,\sigma_u=\sigma_{u_0},\tilde{\sigma}(V_j),\sigma(V_j))$. To be concise, we use the notation
\[
f(u,\tilde{\sigma}(V_j)):=\sum_{v \in V_{j}} \tsigma_v \left[  A_{uv} \log \frac{p}{q} + (1-A_{uv})  \log \frac{1-p Q_{uv}}{1-q Q_{uv}}  \right].
\]
Then $\P_V(\tsigma_u= -1 \,|\,\sigma_u=\sigma_{u_0},\tilde{\sigma}(V_j),\sigma(V_j)) = \P_V(f(u,\tilde{\sigma}(V_j)) <0\,|\,\sigma_u=\sigma_{u_0},\tilde{\sigma}(V_j),\sigma(V_j))$ which can be bounded as
\begin{align}
\P_V(&f(u,\tilde{\sigma}(V_j)) <0\,\big|\,\sigma_u=\sigma_{u_0},\tilde{\sigma}(V_j),\sigma(V_j))\nonumber\\
& \le \E\left[e^{-tf(u,\tilde{\sigma}(V_j))}\,\Big|\,\sigma_u=\sigma_{u_0},\tilde{\sigma}(V_j),\sigma(V_j),V\right]\nonumber\\
&=\E_{\cF_j}\Bigg[\prod_{\stackrel{v \in V_j}{\tsigma_v=+1}}\left(\Big(\frac{q}{p}\Big)^{A_{uv}}\left(\frac{1-qQ_{uv}}{1-pQ_{uv}}\right)^{1-A_{uv}}\right)^t \prod_{\stackrel{v \in V_j}{\tsigma_v=-1}}\left(\Big(\frac{p}{q}\Big)^{A_{uv}}\left(\frac{1-pQ_{uv}}{1-qQ_{uv}}\right)^{1-A_{uv}}\right)^t\Bigg],
\label{eq:huge_expr}
\end{align}
for any $t>0$, where $\cF_j$ is the sigma algebra generated by
$\{\sigma_u=\sigma_{u_0},\tilde{\sigma}(V_j),\sigma(V_j),V\}$,
and $\E_{\cF_j}$ is the conditional expectation given $\cF_j$.
Since given the locations and the true community labels of nodes in $V_j$, the entries $A_{uv}$ are independent,
using the notation introduced in \eqref{eq:error_nbrs}, the probability in \eqref{eq:huge_expr} can be expressed as
\begin{align}
\P_V&(f(u,\tilde{\sigma}(V_j)) <0|\sigma_u=\sigma_{u_0},\tilde{\sigma}(V_j),\sigma(V_j))\nonumber\\
&\le  \prod_{\cZ_{++}(V_j)}\E_{\cF_j}\left[\left(\Big(\frac{q}{p}\Big)^{A_{uv}}\left(\frac{1-qQ_{uv}}{1-pQ_{uv}}\right)^{1-A_{uv}}\right)^t\right] \prod_{\cZ_{-+}(V_j)}\E_{\cF_j}\left[\left(\Big(\frac{q}{p}\Big)^{A_{uv}}\left(\frac{1-qQ_{uv}}{1-pQ_{uv}}\right)^{1-A_{uv}}\right)^t\right] \nonumber\\
& \hspace{0.1cm} \times \prod_{\cZ_{+-}(V_j)}\E_{\cF_j}\left[\left(\Big(\frac{p}{q}\Big)^{A_{uv}}\left(\frac{1-pQ_{uv}}{1-qQ_{uv}}\right)^{1-A_{uv}}\right)^t\right]\prod_{\cZ_{--}(V_j)}\E_{\cF_j}\left[\left(\left(\frac{p}{q}\right)^{A_{uv}}\left(\frac{1-pQ_{uv}}{1-qQ_{uv}}\right)^{1-A_{uv}}\right)^t\right].
\label{eq:prop_error1}
\end{align} 
Taking $t=\frac{1}{2}$ and computing the expectations, we obtain
\begin{align*}
\P_V(f(u,\tilde{\sigma}(V_j))&<0 |\sigma_u=\sigma_{u_0},\tilde{\sigma}(V_j),\sigma(V_j))\\
&\le \prod_{\cZ_{++}(V_j)}\sqrt{pq}Q_{uv}+\sqrt{\left(1-pQ_{uv}\right)\left(1-qQ_{uv}\right)}
\prod_{\cZ_{-+}(V_j)}\left(\frac{q^{3/2}}{p^{1/2}}Q_{uv}+\frac{\left(1-qQ_{uv}\right)^{3/2}}{\left(1-pQ_{uv}\right)^{1/2}}\right)\\
& \hspace{1cm} \prod_{\cZ_{--}(V_j)}\sqrt{pq}Q_{uv}+\sqrt{\left(1-pQ_{uv}\right)\left(1-qQ_{uv}\right)} \prod_{\cZ_{+-}(V_j)}\left(\frac{p^{3/2}}{q^{1/2}}Q_{uv}+\frac{\left(1-pQ_{uv}\right)^{3/2}}{\left(1-qQ_{uv}\right)^{1/2}}\right).
\end{align*}
The products can be written using \Renyi divergences as follows:
\begin{align*}
\P_V&(f(u,\tilde{\sigma}(V_j)) <0|\sigma_u=\sigma_{u_0},\tilde{\sigma}(V_j),\sigma(V_j))\\
&= \exp\bigg[-\frac{1}{2}\Big(\sum_{v\in \cZ_{++}(V_j)}D_{1/2}\left(\Ber(pQ_{uv}),\Ber(qQ_{uv})\right) +\sum_{v\in \cZ_{--}(V_j)}D_{1/2}\left(\Ber(pQ_{uv}),\Ber(qQ_{uv})\right)\Big)\\
&\hspace{1.5cm} +\frac{1}{2}\Big(\sum_{v\in \cZ_{+-}(V_j)}D_{3/2}\left(\Ber(pQ_{uv})\|\Ber(qQ_{uv})\right)+\sum_{v\in \cZ_{-+}(V_j)}D_{3/2}\left(\Ber(qQ_{uv})\|\Ber(pQ_{uv})\right)\Big)\bigg]\\
&= \exp\Bigg[-\frac{1}{2}\Bigg(\sum_{v\in B_j}D_{1/2}\left(\Ber(pQ_{uv}),\Ber(qQ_{uv})\right))\Bigg)\\
&\hspace{1cm} +\frac{1}{2}\Bigg(\sum_{v\in \cZ_{+-}(V_j)}D_{3/2}\left(\Ber(pQ_{uv})\|\Ber(qQ_{uv})\right)+D_{1/2}\left(\Ber(pQ_{uv}),\Ber(qQ_{uv})\right)
\\
&\hspace{1.4cm}+\sum_{v\in \cZ_{-+}(V_j)}D_{3/2}\left(\Ber(qQ_{uv})\|\Ber(pQ_{uv})\right)+D_{1/2}\left(\Ber(pQ_{uv}),\Ber(qQ_{uv})\right)\Bigg)\Bigg].
\end{align*}
Since $\alpha$-\Renyi divergence is monotonically increasing in $\alpha$, it is true that 
\[
D_{1/2}(\Ber(pQ_{uv}),\Ber(qQ_{uv}))
\wle 
\min
\big\{
D_{3/2}(\Ber(qQ_{uv})\|\Ber(pQ_{uv})),
D_{3/2}(\Ber(pQ_{uv})\|\Ber(qQ_{uv}))
\big\}.
\]
Moreover, using $\eps \le Q_{uv} \le 1$ along with \eqref{eq:renyi_div_eval}, the divergence terms can be bounded as
\begin{align}
 D_{3/2}(\Ber(pQ_{uv})\|\Ber(qQ_{uv}))
 &\wle 
 2\log \left[\frac{p^{3/2}}{q^{1/2}}+\frac{(1-p\eps)^{3/2}}{(1-q)^{1/2}}\right] 
 \wle 
 \xi_1(p,q,\eps)
 \nonumber\\
 D_{3/2}(\Ber(qQ_{uv}) \| \Ber(pQ_{uv}))
 &\wle 
 2\log \left[\frac{q^{3/2}}{p^{1/2}}+\frac{(1-q\eps)^{3/2}}{(1-p)^{1/2}}\right]
 \wle 
 \xi_1(p,q,\eps),
 \label{eq:xis}
\end{align}
where $\xi_1(p,q,\eps)$ is defined in \eqref{eq:xi1xi2}. For the other direction, we use 
\begin{align*}
D_{1/2}(\Ber(pQ_{uv}),\Ber(qQ_{uv}))&\ge \text{Hel}^2(\Ber(pQ_{uv}),\Ber(qQ_{uv})) \\
&= (\sqrt{pQ_{uv}}-\sqrt{qQ_{uv}})^2+(\sqrt{1-pQ_{uv}}-\sqrt{1-qQ_{uv}})^2\\
& \ge \eps (\sqrt{p}-\sqrt{q})^2 = \xi_2(p,q,\eps).
\end{align*}
Using these definitions and further conditioning on the number of errors in segment $B_j$ to be at most a constant, i.e., $|\cZ_{+-}(V_j)|+|\cZ_{-+}(V_j)| \le M$, we can write
\begin{align*}
\P_V(&f(u,\tilde{\sigma}(V_j)) <0|\sigma_u=\sigma_{u_0},\tilde{\sigma}(V_j),\sigma(V_j),|\cZ_{+-}(V_j)|+|\cZ_{-+}(V_j)| \le M)\\
&\le 
\exp\Bigg[-\frac{1}{2}\Bigg(\sum_{v\in V_j}D_{1/2}\left(\Ber(pQ_{uv}),\Ber(qQ_{uv})\right)\Bigg)+\frac{1}{2}\Bigg(\sum_{v\in \cZ_{+-}(V_j)}\xi_1+\xi_1+\sum_{v\in \cZ_{-+}(V_j)}\xi_1+\xi_1\Bigg)\Bigg]\\
&\le 
\exp\Bigg[-\frac{1}{2}\Bigg(\sum_{v\in V_j}D_{1/2}\left(\Ber(pQ_{uv}),\Ber(qQ_{uv})\right)\Bigg)+\xi_1\big(|\cZ_{+-}(V_j)|+|\cZ_{-+}(V_j)|\big)\Bigg]\\
&\le 
\exp\left[-\frac{1}{2}\xi_2 |V_j|\right]e^{\xi_1 M}
\end{align*}
Since $|V_j|>\delta \log n$, we obtain 
\begin{equation}
 \P_V\Big( f(u,\tilde{\sigma}(V_j)) <0 \, \Big| \, \sigma_u=\sigma_{u_0},\tilde{\sigma}(V_j),\sigma(V_j),|\cZ_{+-}(V_j)|+|\cZ_{-+}(V_j)| \le M \Big)
 \wle 
 e^{\xi_1 M} n^{-\frac{\delta \xi_2}{2}}
 \weq 
 \eta_1 n^{-c'}.
 \label{eq:error_pos}
\end{equation}
In a similar way, we can also obtain
\begin{equation}
 \P_V\Big( f(u,\tilde{\sigma}(V_j)) \ge 0 \, \Big| \, \sigma_u\neq \sigma_{u_0},\tilde{\sigma}(V_j),\sigma(V_j),|\cZ_{+-}(V_j)|+|\cZ_{-+}(V_j)| \le M \Big)
 \wle 
 \eta_1 n^{-c'}.
\label{eq:error_neg}
\end{equation}
From \eqref{eq:error_pos} and \eqref{eq:error_neg}, conditioning on whether $u$ and $u_0$ are in the same community or not proves the statement of the proposition.
\end{proof}

\begin{rem}
    The statement of Proposition \ref{prop:prop_single_error} holds for any constant $M$ and, in particular, also to the constant  defined in \eqref{eq:M}.
\end{rem}

\subsubsection{Step 2: Number of mistakes in each segment}
In this section, we show that there are at most a constant number of errors in each of the occupied segments.
For $j=1,\dots,J$, let $\cA_j$ be the event that the propagation step makes at most $M$ errors in segment $B_j$, i.e.,
\begin{equation}
\label{event:M_errors}
\cA_j 
= \{\Ham_{B_j}(\tbsigma,\bsigma) \le M\},
\end{equation}
and $\cI_j$ be the event 
\begin{equation}
\cI_j 
= \{\delta \log n \le |V_{j}| \le \Delta \log n\}.
\label{eq:Ij_event}
\end{equation}
 Note that $\P_V(\cA_1) \ge \P_{V_1} \left(
\Ham_{B_1}(\tbsigma, \bsigma)=0\right) \ge (1-\Delta n^{-c\de \eps^4} \log n)$ from Lemma~\ref{lem:error_all_nodes_phase1}. The following lemma characterizes the total number of errors made in a single segment $B_j$ for $j\ge 2$.

\begin{lem}
For any $j=1,\cdots,J-1$, 
\[
\P_V \left(
\cA_{j+1}^c
\ \Big| \
 \, \tilde{\sigma}(V_{j}),\sigma(V_j),\cA_j^c,\cI_j
\right)
\wle \left(\frac{e\Delta\eta_1}{M}\right)^M  n^{-9/8}.
\]

\label{lem:Hamming}
\end{lem}
\begin{proof}
Since the estimate $\tsigma_u$ for $ u\in V_{j+1}$ is independent for each node conditional on the previous occupied segment, the number of errors in each segment can be stochastically dominated by a binomial random variable 
\[
 \Ham_{B_j}(\tbsigma,\bsigma)
 \preccurlyeq \Bin(\Delta \log n,\eta_1 n^{-c'})
 \triangleq Z,
\]
with mean $\mu_Z$. The required probability can then be bounded as 
\begin{align*}
 \P_V\left(\cA_{j+1}^c \ \big|\ \tilde{\sigma}(V_{j}), \sigma(V_j), \cA_{j} ,\cI_j\right) 
 &\wle \P(\Bin(\Delta \log n,\eta_1 n^{-c'})> M )\\
 &\weq \P(Z-\mu_Z > M-\mu_Z)\\
 &\weq \P\left(Z > \mu_Z\left(1+\frac{M-\mu_Z}{\mu_Z}\right)\right).
\end{align*}
Using Lemma~\ref{lem:conc_binomial} on the concentration of the binomial distribution, we obtain
\begin{align*}
 \P_V\left(\cA_{j+1}^c \ \big|\ \tilde{\sigma}(V_{j}) ,\sigma(V_j) , \cA_{j},\cI_j\right)
 &\le \frac{\big(e^{\frac{M-\mu_Z}{\mu_Z}}\big)^{\mu_Z}}{\bigg(\left(\frac{M}{\mu_Z}\right)^{\frac{M}{\mu_Z}}\bigg)^{\mu_Z}}\\
 &\le e^{M-\mu_Z} \left(\frac{\mu_Z}{M}\right)^M\\
 &\le \left(\frac{e\Delta\eta_1}{M}\right)^M \frac{(\log n)^M}{n^{c'M}},
\end{align*}
since $e^{-\mu_Z} \le 1$.
Note that $c'= \frac{\delta \xi_2}{2}$ which gives $c'M = \frac{\delta M \eps (\sqrt{p}-\sqrt{q})^2}{2} = \frac{10}{8}$. Along with $(\log n)^M \le n^{1/8}$ for large enough $n$, we obtain the statement in the lemma
\end{proof}

\subsubsection{Step 3: Almost exact recovery}
The final step of the propagation phase involves showing that the estimate in Line~\ref{line:after_prop} of Algorithm~\ref{alg:full} recovers the communities almost exactly. Along with nodes in segments that are not $\delta$-occupied, we show that there are at most $\eta \log n$ number of errors in the vicinity of every node for some $\eta>0$. The estimate is cleaned up to remove these errors in the refinement phase.

Let $\cG = \cH \cap \cI$, where $\cH$ is the event
the $\de$-skeleton is $(\ka,\chi)$-connected,
and
$\cI$ is the event that $\max_i \abs{V(C_i)} \le \De \log n$.
In particular, any point configuration in $\cG$ satisfies $\delta \log n \le |V_{j}| \le \Delta \log n, j =1,\dots,J$ for occupied segments, and therefore $\cG \subset \bigcap_{j=1}^J \cI_j$. Then, using Lemma~\ref{lem:max_nodes} and Lemma~\ref{lem:connected_vis_graph}  we have $\P(\cG) \ge 1-\frac{2}{\chi \log n}$. 
We now evaluate the effectiveness of the propagation phase in the following lemma.

\begin{lem}
\label{lem:prop_eventA}
Let $\la>0$, $0<q<p<1$, $0 < \ka < \infty$, and assume that $\phi(x)>0$ for all $x\in[0,\ka]$. If $\la \ka>1$, for any realisation of $V \in \cG$, the output $\tbsigma$ of the propagation phase of Algorithm~\ref{alg:full} satisfies 
\[
\P_V\bigg(\max_{1 \le j \le J} \Ham_{B_j}( \tbsigma, \bsigma )
\le M \bigg)
\wge 1-o(1),
\]
where $M$ is as defined in \eqref{eq:M}.
\end{lem}
\begin{proof}

Recall the definition of $\cA_j$ in \eqref{event:M_errors}. For any realisation of $V$ such that $\delta \log n \le |V_{j}| \le \Delta \log n$ for $j=1,\cdots,J-1$, the probability $\P_V(\cA_{j+1}^c|\ \tilde{\sigma}(V_{j}), \sigma(V_j), \cA_{j}) = \P_V(\cA_{j+1}^c|\ \tilde{\sigma}(V_{j}), \sigma(V_j),\bigcap_{k<j}\cA_k)$ since given the locations and estimated communities of nodes in $V_j$, the estimates $f(u,\tsigma(V_{j}))$ are independent of $\cA_k$ for $k<j$. Moreover, since the bound from Lemma~\ref{lem:Hamming} does not depend on $\tsigma(V_j)$ and $\sigma(V_j)$ we can uniformly bound the probability as $\P_V(\cA_{j+1}^c|\ \bigcap_{k<j} \cA_{k}) \le \eta_2 n^{-9/8}$, where $\eta_2 = \left(\frac{e\Delta\eta_1}{M}\right)^M$.
Thus, we obtain
\begin{align*}
\P_V\left(\bigcap_{j=1}^{J} \cA_j \right)
&= \P_V(\cA_1) \prod_{j=2}^{J}\P_V\bigg(\cA_j\Big| \bigcap_{k<j}\cA_k\bigg) \\
&\ge \left(1-\Delta n^{-c\de \eps^4} \log n\right) \left(1-\eta_2 n^{-9/8}\right)^{\frac{n}{\chi \log n}} \\
&\ge \left(1-\Delta n^{-c\de \eps^4} \log n\right) \left(1-\frac{\eta_2}{\chi n^{1/8}\log n} \right)\\
&= 1-o(1).
\end{align*}
\end{proof}
Recall the definition of a visibility set in Definition \ref{defn:visibility_set}.
The following theorem asserts that the community estimate $\tilde{\bsigma}$ obtained after the initialization and the propagation phases recovers the node communities almost-exactly.

\begin{thm}
\label{thm:almost_exact}
Let $\la>0$, $0 < \zeronorm{\phi} < \infty$, and assume that $\phi(x)>0$ for all $x\in[0,\zeronorm{\phi}]$. If $\la \zeronorm{\phi}>1$, then $\tbsigma$ recovers the communities almost exactly as defined in \eqref{eq:almost_exact_recovery}. Moreover, for any $\eta >0$,
\begin{align*}
\P\Big(\max_{u\in V} \  \Ham_{\cV(u)}(\tbsigma,\bsigma) \le \eta \log n \Big) 
\ge 1-o(1).
\end{align*} 
\end{thm}
\begin{proof}
Fix any $\eta >0$. Let $\chi$ be chosen as in \eqref{eq:chi}.
Choose $\delta$ according to \eqref{eq:delta} and satisfying $\delta < \frac{\eta \chi}{2\ka+\chi}$.
From Lemma~\ref{lem:prop_eventA}, there exists a constant $M$ such that
$
\P_V\left(\bigcap_{j=1}^{J} \{\Ham_{B_j}(\tbsigma,\bsigma) \le M\}\right) \ge 1-o(1),
$
for any realization $V\in \cG$. Note that this is a uniform bound on the probability independent of the realization. Moreover, since $M< \delta \log n$ for sufficiently large $n$, 
$\P_V\left(\bigcap_{j=1}^{J} \{\Ham_{B_j}(\tbsigma,\bsigma) \le \delta \log n\}\right) \ge 1-o(1).$
For a segment $C_i$ that is not $\delta$-occupied, since $|V(C_i)| \le \delta\log n$, we obtain
\begin{align*}
\P_V\left(\bigcap_{i=1}^{\frac{n}{\chi\log n}} \{\Ham_{C_i}(\tbsigma,\bsigma) \le \delta \log n\}\right) 
\ge 1-o(1),
\end{align*}
for any $V\in \cG$. To show that $\tbsigma$ recovers $\bsigma$ almost exactly, we bound the required probability in \eqref{eq:almost_exact_recovery} as follows. Let $\eta' = \frac{ \eta }{2 \ka+\chi}$. Using Remark \ref{rem:Hamming_dist_relation}, we obtain
\begin{align*}
\P\big(\Ham(\tbsigma,\bsigma) \le \eta'n\big)  
&\wge \P\big(\Ham_V(\tbsigma,\bsigma) \le \eta'n \,\big| \,\cG\big)\ \P(\cG)\\
&\wge \P\bigg(\bigcap_{i=1}^{\frac{n}{\chi\log n}} \{\Ham_{C_i}(\tbsigma,\bsigma) \le \delta \log n\}\,\bigg| \,\cG\bigg) \ \P(\cG)\\
&\wge 1-o(1),
\end{align*}
where the second inequality is obtained since if $\tilde{\bsigma}$ makes fewer than $\delta \log n $ mistakes within each segment $C_i$, then it makes at most $\frac{n}{\chi \log n} \delta \log n  < \frac{n \eta}{2 \ka+\chi}$ mistakes on the whole. Since $\eta$ was arbitrary, the estimate $\tilde{\bsigma}$ recovers the communities almost exactly. 

For the second part of the theorem, note that since for every $u\in V$ the nodes in the visibility set $\cV(u)$ can be in at most $\frac{2 \ka}{\chi}+1$ segments, the number of mistakes among them can be at most $\Big(\frac{2 \ka}{\chi}+1\Big) \delta \log n < \eta \log n$. Thus, we have
\begin{align*}
\P\bigg(\bigcap_{u\in V}\Big\{ \Ham_{\cV(u)}(\tbsigma,\bsigma)  \le \eta \log n \Big\}\bigg) \ge 1-o(1).
\end{align*} 
\end{proof}

\subsection{Refinement phase}
Lines \ref{line:refine_start}--\ref{line:refine_end} of Algorithm~\ref{alg:full} refine the estimate $\tilde{\bsigma}$ obtained after the propagation phase to recover the ground truth communities up to a global sign flip. In this section we obtain a concentration bound on the quantity 
\[
g(u,\bsigma)
:= \sum_{v\in \cV(u)} \sigma_v \left[
A_{uv} \log \frac{p}{q} + (1-A_{uv})
\log \frac{1-p Q_{uv}}{1-q Q_{uv}} \right],
\]
where $\cV(u)$ is the visibility set of $u\in V$ (see Definition \ref{defn:visibility_set}). This is in turn used to prove Theorem~\ref{thm:exact_recovery}.

\begin{prop}
\label{prop:refinement_conc_bound}
For any $\eta'>0$,
\begin{align*}
\P(g(u,\bsigma) \ge - \eta' \log n \,|\,\sigma_u=-1) 
&\wle n^{\big[\frac{\eta'}{2} - \la I_{\phi}(p,q)\big]},\\
\P(g(u,\bsigma) <\eta' \log n \,|\,\sigma_u=+1) 
&\wle n^{\big[\frac{\eta'}{2} - \la I_{\phi}(p,q)\big]}.
\end{align*}
\end{prop}
\begin{proof}
Note that $g(u,\bsigma) = -\sum_v R_{uv}$ where $R_{uv}$ is introduced in \eqref{eq:ruv}. Similar to Section \ref{sec:first_moment}, we introduce the notation $\P_{-} (\P_+)$ and $\E_{-} (\E_+)$ for the probability and expectation conditioned on $\sigma_u = -1$ (resp., $\sigma_u = +1$). Using the Chernoff bound we obtain
\begin{align}
\P_{-}(g(u,\bsigma) \ge - \eta' \log n) 
&\le \exp\Big[t\eta' \log n+\log \E_{-}\big[e^{tg(u,\bsigma)} \big]\Big] \nonumber\\
&= \exp\Big[t\eta' \log n+\log \E_{-}\big[e^{-t\sum_v R_{uv}} \big]\Big].
\label{eq:mgf3}
\end{align}
The term $\E_{-}\big[e^{-t\sum_v R_{uv}} \big]$ is the moment generating function of a compound Poisson process since the sum is over the visible set of the vertex $u$. This evaluates to 
\begin{equation}
\E_{-}\big[e^{-t\sum_v R_{uv}} \big] 
= \exp \Big[2 \la \ka \log n \big(\E_{-}\big[\exp\big(-tR_{uv}\big)\big]-1\big)\Big].
\label{eq:mgf_sum}
\end{equation}

Computing the moment generating function of $R_{uv}$ similar to \eqref{eq:mgf2}, we obtain
\begin{align*}
\E_-\Big[\exp(-t R_{uv})\Big]
&= \frac{1}{2\ka} \int_0^{\ka } \Big[(p\phi(z))^t (q\phi(z))^{1-t}+(1-p\phi(z))^t(1-q\phi(z))^{1-t} \\
&\hspace{3.5cm}+(q\phi(z))^t (p\phi(z))^{1-t}+(1-q\phi(z))^t(1-p\phi(z))^{1-t}\Big] dz.
\end{align*}
Taking $t=1/2$ gives
\begin{align}
\E_-\Big[\exp(- R_{uv}/2)\Big]
&= \frac{1}{\ka} \int_0^\ka \Big[\sqrt{pq}\phi(z)+\sqrt{(1-p\phi(z))(1-q\phi(z))}\Big] dz.
\label{eq:mgf_ruv}
\end{align}
Substituting \eqref{eq:mgf_sum} and \eqref{eq:mgf_ruv} in \eqref{eq:mgf3},
\begin{align}
\P_{-}(g(u,\bsigma) \ge -\eta' \log n) 
&\le  \exp \Big[\frac{\eta'}{2} \log n +2 \la \ka \log n \big( \frac{1}{\ka} \int_0^\ka \Big[\sqrt{pq}\phi(z)+\sqrt{(1-p\phi(z))(1-q\phi(z))}\Big] dz-1\big)\Big] \nonumber\\
&= \exp \bigg[\frac{\eta'}{2} \log n +2 \la \ka \log n \Big( 1-\frac{I_{\phi}(p,q)}{2\ka}-1\Big)\bigg] \nonumber\\
&=n^{\big[\frac{\eta'}{2} - \la I_{\phi}(p,q)\big]}.
\label{eq:minus_error}
\end{align}
Similarly, conditioned on $\sigma_u=+1$, we obtain the moment generating function at $t=1/2$ to be
\begin{align}
\E_+\Big[\exp(R_{uv}/2)\Big]
&= \frac{1}{\ka} \int_0^\ka \Big[\sqrt{pq}\phi(z)+\sqrt{(1-p\phi(z))(1-q\phi(z))}\Big] dz,
\label{eq:mgf_ruv_plus}
\end{align}
giving 
\[
\P(g(u,\bsigma) <\eta' \log n \,|\,\sigma_u=+1) 
\wle n^{\big[\frac{\eta'}{2} - \la I_{\phi}(p,q)\big]}.
\]
\end{proof}

\subsection{Proof of Theorem~\ref{thm:exact_recovery}}
Let $\tbsigma$ be the output from Line~\ref{line:after_prop} of Algorithm~\ref{alg:full}.
To prove the correctness of the refinement phase, a natural way to proceed is to show that the probability of error that the algorithm makes in assigning the community to a single node is $o(\frac{1}{n})$ and then use a union bound. However, since there are a random (Poisson) number of nodes and the statistics $g(u,\hat{\bsigma})$ are dependent we use an alternate procedure that is detailed in \cite{gaudio2024exact}.

For this fix a $c>\la$ and let $\cG_0 = \{|V|<cn\}$. Since $|V|\sim \text{Poi}(\la n)$, using the Chernoff bound from Lemma~\ref{lem:conc_poisson} we have that 
$$\P(\cG_0^c) 
\le \exp \left[ -\frac{(c-\la)^2 n}{2c}\right] 
= o(1).$$ 
For (still to be determined) $\eta >0$, let $\cG_1$ be the event that for every node $u\in V$, the estimate $\tilde{\bsigma}$ makes at most $\eta \log n$ mistakes in the visibility set $\cV(u)$, i.e.,
$$\cG_1 =\cap_{u\in V}\{\Ham_{\cV(u)}(\tbsigma, \bsigma) \le \eta \log n\}.$$ From Theorem~\ref{thm:almost_exact}, we have that $\P(\cG_1^c) =o(1)$. From Remark \ref{rem:Hamming_dist_relation},
our interest is in bounding the  probability of the error event 
$\cE = \cup_{u\in V} \{\hsigma_u  \ne \sigma_u \sigma_{u_0}\}$. Note that 
\begin{equation}
\P(\cE) 
\wle \P(\cE \cap \cG_1 \cap \cG_0 ) + \P(\cE \cap \cG_1^c)+ \P(\cE \cap \cG_0^c)
\weq \P(\cE \cap \cG_1 \cap \cG_0 )+o(1).
\label{eq:error_e2}
\end{equation}

To address the term on the RHS, we couple the original model with another model in which number of nodes is deterministic.
First sample an integer $N \sim \Poi(\la n)$, and let $N' = \max\{N,cn\}$.
Then sample points $v_1,\dots,v_{N'}$ independently and uniformly at random in $(-n/2,n/2]$,
and denote $V = \{v_1,\dots,v_N\}$ and $V' = \{v_1,\dots,v_{N'}\}$.
Let $\bsigma' \colon V' \to \{-1,+1\}$ be sampled independently and uniformly at random.
Let $A'_{uv}$ be Bernoulli random variables with mean \eqref{eq:InteractionProduct} for all $u,v\in V'$.
Now $(V', \bsigma', \bA')$ constitutes a sample from an extended $\GKBM$ model.
Let $\bA'_{V,V}$ be the submatrix of $\bA'$ restricted to nodes in $V$,
and let $\bsigma'_V$ be the restriction of $\bsigma'$ to nodes in $V$.
Now we see that $(V, \bsigma'_V, \bA'_{V,V})$ is a sample from the original $\GKBM_n(\la,\phi,p,q)$ model.

Let $\hbsigma$ (resp.\ $\tbsigma$) be the output of the full (resp. only the Initialization and Propagation phases of) Algorithm~\ref{alg:full} on $(V,\bA'_{V,V})$. Define
$\tbsigma' \in \{-1, 0, +1\}^{V'}$ by
\[
\tsigma'_v
\weq 
\begin{cases}
\tsigma_v, &\quad v \in V, \\
0, &\quad \text{else}.
\end{cases},
\]
and $\hbsigma' \in \{\pm 1\}^{V'}$ by
\[
\hsigma'_u
\weq \sign( g(u, \tbsigma') ).
\]
Because $\tsigma'_u = 0$ for $u \notin V$, we see that the labels of the
auxiliary vertices $\{\tsigma'_u: u\in V'\setminus V\}$ do not affect the refined estimates $\hbsigma$ of nodes in $V$, so that
$\hsigma'_u = \hsigma_u$ for all $u \in V$.
It follows that
\begin{equation}
\P(\cE \cap \cG_1 \cap \cG_0 )
\le \sum_{u \in [cn]} \P( \{\hsigma'_u \ne \sigma'_u \sigma'_{u_0}\} \cap \cG_1 ).
\label{eq:error_e2_simp}
\end{equation}
Note that on the RHS of \eqref{eq:error_e2_simp} we have a model with a deterministic number of nodes and we wish to obtain the refined estimates $\hsigma'_u$ for all $u\in [cn]$ based on edges and non-edges to nodes in $V$. 

Let $W(u) = \{\tbsigma' : \cV(u) \rightarrow \{-1,0,+1\}\}$ be the set of community assignments on $\cV(u)$.  Additionally, note that for node $u \in [cn]$, $g(u,\tbsigma')$ depends only on the nodes in $\cV(u)$. Hence, for a fixed $u$, we can think of the quantity $g$ as a function with inputs being the node $u$ and the communities of nodes within the visibility set of $u$. In other words, $g(u,\tbsigma') \equiv g(u,\tbsigma'_{\cV(u)})$. We will use this notation in the following discussion. Let $W'(u;\eta)$ be the set of all community estimates that differ from the ground truth $\bsigma'$ on at most $\eta \log n$ nodes within $\cV(u)$, i.e.,
$$
W'(u;\eta) 
= \{\tbsigma' \in W(u): \Ham_{\cV(u)}(\tbsigma', \sigma'_{u_0}\bsigma') \le \eta \log n\}
= \{\tbsigma' \in W(u): \Ham_{\cV(u)}(\sigma'_{u_0}\tbsigma', \bsigma') \le \eta \log n\}.
$$
Consider a node $u\in[cn]$ such that $\sigma'_u = +1$. If node $u$ is assigned to the wrong community, then there must be at least one labeling $\tbsigma' \in W'(u;\eta)$ for which $g(u,\tbsigma') < 0$. A similar reasoning holds when $\sigma'_u= -1$. If we now define
\begin{align*}
\cE'_u 
:= \Big\{\{\sigma'_u = 1\} \cap \left\{\cup_{\tbsigma'\in W'(u;\eta)} \{g(u,\sigma'_{u_0}\tbsigma') < 0 \}\right\}\Big\}
\bigcup \Big\{\{\sigma'_u = -1\} \cap \left\{\cup_{\tbsigma'\in W'(u;\eta)}\{ g(u,\sigma'_{u_0}\tbsigma') \ge 0 \}\right\}\Big\},
\end{align*}
we have that $\P( \{\hsigma'_u \ne \sigma'_u  \sigma'_{u_0}\} \cap \cG_1 ) \le \P(\cE'_u)$, and from \eqref{eq:error_e2} and \eqref{eq:error_e2_simp} we obtain 
\begin{equation}
\P(\cE) \le \sum_{u=1}^{cn} \P(\cE'_u).
\label{eq:errorE2_union}
\end{equation}
Conditioning on the community of node $u$, we have
\begin{align}
\P(\cE'_u)
&\weq \frac{1}{2}\Big[\P(\cE'_u \, | \, \sigma'_u=-1)+\P(\cE'_u \, | \, \sigma'_u=+1)\Big] \nonumber\\
&\weq \frac{1}{2}\Big[\P(g(u,\tbsigma') \ge0 \,|\, \sigma'_u = -1)+\P(g(u,\tbsigma') <0\,|\, \sigma'_u = +1) \Big].
\label{eq:error_Eu}
\end{align}
We bound these probabilities by assuming that the initialization and propagation phases outputs the worst case estimate $\tbsigma'$. To go about this we obtain a bound on the difference $|g(u,\tbsigma') - g(u,\bsigma')|$ using the definition of $W'(u;\eta)$ as follows:
\begin{align}
|g(u,\sigma'_{u_0}\tbsigma')
&- g(u,\bsigma')| \nonumber\\
&= \Bigg|\sum_{\substack{v: A'_{uv}=1 \\ \tsigma'_v=\sigma'_{u_0}\\ \sigma'_v=-1}} 2\log \frac{p}{q} 
+ \sum_{\substack{v : A'_{uv}=0 \\ \tsigma'_v=\sigma'_{u_0}\\ \sigma'_v=-1}} 2\log \frac{1-pQ_{uv}}{1-qQ_{uv}} 
+ \sum_{\substack{v: A'_{uv}=1 \\ \tsigma'_v\neq\sigma'_{u_0}\\ \sigma'_v=+1}} 2\log \frac{q}{p} 
+ \sum_{\substack{v: A'_{uv}=0\\ \tsigma'_v\neq\sigma'_{u_0}\\ \sigma'_v=+1}} 2\log \frac{1-qQ_{uv}}{1-pQ_{uv}}\Bigg| \nonumber\\
&\le \Bigg|\left(2\log \frac{p}{q}+2\log \frac{1-p\eps}{1-q} \right) |\{v\in\cV(u):\tsigma'_v=\sigma'_{u_0}, \sigma'_v=-1\}|\\
& \hspace{2cm}+  \left(2\log \frac{q}{p}  +2\log \frac{1-q\eps}{1-p} \right)|\{v\in\cV(u):\tsigma'_v\neq \sigma'_{u_0}, \sigma'_v=+1\}|\Bigg|\nonumber \\
&\le \beta_\eps \eta \log n
\label{eq:refine_diff}
\end{align}
where $\beta_\eps := \Big|2\log \frac{p}{q}+2\log \frac{1-p\eps}{1-q}+2\log \frac{q}{p}  +2\log \frac{1-q\eps}{1-p} \Big|$. Thus the worst case estimate $\bsigma'$ is such that 
$$g(u,\bsigma') -\beta_\eps \eta \log n \le  g(u,\tbsigma') \le g(u,\bsigma') +\beta_\eps \eta \log n. $$
Using \eqref{eq:refine_diff}, the first term on the RHS in \eqref{eq:error_Eu} can be written as
\begin{align}
\P(g(u,\tbsigma') \ge 0\,|\, \sigma_u = -1) 
\wle \P(g(u,\bsigma') \ge -\beta_\eps \eta\log n\,|\, \sigma_u = -1)
\label{eq:minus_error_simp}
\end{align}
Similarly, conditioned on $\sigma_u=+1$, 
\begin{align}
\P(g(u,\tbsigma') <0|\sigma_u = +1) 
&\wle \P(g(u,\bsigma') <\beta \eta \log n\,|\,\sigma_u = +1).
\label{eq:plus_error}
\end{align}
Using Proposition \ref{prop:refinement_conc_bound} with $\eta'=\beta \eta $, along with \eqref{eq:plus_error}, \eqref{eq:minus_error_simp} and \eqref{eq:error_Eu} we get
\begin{align}
\P(\cE) 
&\le c n^{\big[1- \la I_{\phi}(p,q)+\frac{\beta \eta}{2} \big]}.
\end{align}
Since $\la I_{\phi}(p,q) >1$, choosing $\eta = \frac{\la I_{\phi}(p,q) -1}{\beta}$ yields $\P(\cE) \le n^{\big[\frac{1- \la I_{\phi}(p,q)}{2} \big]}=o(1) $. 
This proves the correctness of the refinement phase and shows exact recovery when $\la I_{\phi}(p,q) >1$.

\section{Conclusions}
\label{sec:conclusions}
In this work, we consider the problem of community recovery on block models in which edges are present based on the community of nodes as well as their geometric position in a Euclidean space. The dependence on the communities is through the intra-community and inter-community connection parameters $p$ and $q$ respectively, and the dependence on the underlying Euclidean space is via a geometric kernel $\phi$. For the one-dimensional case with two communities, we have obtained conditions on the model parameters $p,q,\phi$ for which no algorithm can recover the communities exactly. Additionally, we have provided a linear-time algorithm that guarantees recovery up to the information-theoretic threshold. Our techniques for the information-theoretic criterion (Section \ref{sec:info_metric}) extend to higher dimensions and larger number of communities as well. We also believe that our algorithm could be extended to higher dimensions by propagating over a spanning tree on the segments as in \cite{gaudio2024exact}. This constitutes an important topic for future work.  Another direction for future research is to extend the algorithm to cases when the parameters of the model are not known.

%
%

\bibliographystyle{abbrv}
\bibliography{references}

\begin{appendices}

\section{Some useful concentration bounds}

In this section, we provide some useful concentration bounds. These can be obtained from standard texts such as \cite{boucheron2003concentration}.

\begin{lem}[Hoeffding's inequality]
Let $X_1 ,\cdots , X_n$ be independent random variables such that $X_i$ takes its values in $[a_i , b_i ]$ almost surely for all $i \le n$. Let
$S = \sum_{i=1}^n (X_i-\E[X_i])$. Then for every $t>0$,
\[
 \P(|S|\ge t)
 \wle 2\exp\bigg[-\frac{2t^2}{\sum_{i=1}^n (b_i-a_i)^2}\bigg].
\]    
\label{lem:hoeffding}
\end{lem}

\begin{lem}[Chernoff bound for Poisson random variables]
\label{lem:conc_poisson}
Let $X$ be Poisson-distributed with mean $\mu > 0$.
Then
\[
 \P(X \ge t)
 \wle e^{-\mu h(t/\mu)} \wle \exp\left[-\frac{(t-\mu)^2}{2t}\right]
 \qquad \text{for all $t \ge \mu$},
\]
and
\[
 \P(X \le t ) 
 \wle e^{-\mu h(t/\mu)}
 \qquad \text{for all $0 < t < \mu$},
\]
where $h(x) = x \log x + 1 - x$.
\end{lem}

\begin{lem}[Chernoff bound for binomial random variables]
Let $X \sim \Bin(n,p)$ with mean $\mu=np$.
For any $t>0$, we have
\[
 \P(X \ge \mu(1+t))
 \wle \left(\frac{e^t}{(1+t)^{(1+t)}}\right)^\mu.
\]
\label{lem:conc_binomial}
\end{lem}

\begin{lem}
\label{lem:conc_poisson_new}
Let $X$ be Poisson-distributed with mean $\mu \log n$. Then for any $0 < \delta < \mu$,
\[
 \P( X \le \delta \log n )
 \wle n^{- \mu h(\de/\mu)},
\]
where
$h(x) = x \log x + 1 - x$.
Furthermore, if $0 < \al < \mu-1$, then
\[
 \P( X \le \be \log n )
 \wle n^{-(1 + \al)}
\]
with $\be = \mu h^{-1}( \frac{1+\al}{\mu} )$,
where $h^{-1}(\cdot)$ denotes the inverse of $h(\cdot)$ on $(0,1)$.
\end{lem}

\begin{proof}
The first part of the lemma is a direct consequence of the bound on the lower tail of a Poisson random variable in Lemma~\ref{lem:conc_poisson}.

Assume next that that $0 < \al <\mu-1$.
Then $\frac{1+\al}{\mu} \in (0,1)$. Because $h$ is a strictly decreasing
bijection from $(0,1)$ onto $(0,1)$, we may define
$\be = \mu h^{-1}( \frac{1+\al}{\mu} )$.
Then $h(\be/\mu) = \frac{1+\al}{\mu}$,
and the second claim follows from the first.
\end{proof}

\section{MAP is Bayes optimal }
\label{app:MAP_Bayes_opt}
In this section, we show that the MAP estimate defined in \eqref{eq:MAP} is Bayes optimal. While it is a well-known result (see \cite[Section 5.7.1]{murphy2012machine}) that the MAP estimate is Bayes optimal for the $0$-$1$ loss or the Hamming loss, we were unable to locate a reference that shows the same result for the permutation invariant Hamming loss defined in \eqref{eq:Hamming_dist}. 

We now extend this result to the case of the permutation invariant Hamming distance.
To provide a general result, in the following, we consider $n$ nodes in $K$ communities with community assignment $\bsigma = (\sigma_1,\sigma_2,\cdots,\sigma_n) \in [K]^n$. Let $\cS_K$ is the permutation group on $K$ elements. For any $\pi \in \cS_K$, $\pi \circ \bsigma = (\pi(\sigma_1),\pi(\sigma_2),\cdots,\pi(\sigma_n))$. For any two community vectors $\bsigma, \btau \in [K]^n$, define a relation
\begin{equation}
    \label{eq:relation}
    \bsigma \sim \btau \text{ iff } \exists \pi \in \cS_K \text{ such that } \pi\circ \btau = \bsigma.
\end{equation}
\begin{claim}
    The relation $\sim$ defined in \eqref{eq:relation} is an equivalence relation.
\end{claim}
\begin{proof}
    The reflexive property holds with the identity permutation. The symmetric property holds since if $\bsigma = \pi \circ \btau$ for some $\pi \in \cS_K$, then $\btau = \pi^{-1} \circ \bsigma$. Finally, the transitive property is also satisfied since if $\bsigma = \pi_1 \circ \btau_1$ and $\btau_1 = \pi_2 \circ \btau_2$ for any $\btau_1,\btau_2 \in [K]^n$, then $\bsigma = (\pi_1\circ \pi_2)\circ \btau_2$.
\end{proof}
Take $\zeta = \{\Theta_1,\Theta_2,\cdots\}$ to be the set of all equivalence classes of the relation $\sim$ defined in \eqref{eq:relation}. Each equivalence class assimilates all community assignments that differ by a permutation of the labels. Denote a generic element of this set by $\Theta$. Given a parameter $\theta \in \Theta$, a graph $G$ is generated on the $n$ vertices from a distribution $P_\theta$. We make the following assumption on the distributions $\{P_\theta,\theta \in \Theta\}$:
\begin{assumption}
    The distributions $\{P_\theta,\theta \in \Theta\}$ are permutation invariant, i.e., $P_\theta = P_{\pi \circ \theta}$ for any $\pi \in \cS_K$.
    \label{ass:prob_perm_invariant}
\end{assumption}
Consider the estimation problem of recovering the equivalence class $\Theta$ by observing the graph $G$ under the $0$-$1$ loss $L(\hat{\Theta},\Theta) = \1_{\{\hat{\Theta} \neq \Theta\}}$. Being a point estimation problem, it is known that the MAP estimate $\Tmap = \arg \max_{\Theta \in \zeta} \P(\Theta | G)$ minimizes the posterior expected loss which is to say
\[
\Tmap = \arg\min_{\Theta' \in \zeta}\ \E\big[L(\Theta', \Theta)\big], \quad \text{and therefore }\quad \P(\Tmap \neq \Theta) \weq  \min_{\Theta' \in \zeta}\ \P(\Theta'\neq \Theta).
\]
Here $\P$ denotes the posterior distribution.
Specializing this to our situation, note that the event $\{\Theta' \neq \Theta\}$ means that the corresponding equivalence classes are different. Since the equivalence classes are disjoint, it should not be possible to obtain an estimate $\theta' \in \Theta'$ via any permutation $\pi \circ \theta $ for $\theta \in \Theta$ when $\Theta \neq \Theta'$. This corresponds to $\Ham(\theta',\theta) > 0$. In the case of $K=2$ communities labelled $\{-1,+1\}$, this can simply be written as $\P(\theta' \not \in \{\theta,-\theta\})$ as done in \eqref{eq:MAP_Bayes_opt}. Note that Assumption \ref{ass:prob_perm_invariant} is necessary in order for the distributions associated with an equivalence class to be the same. This is satisfied in our case since the connections depend only on whether two nodes are within the same community or not. However, for multiple communities, Assumption \ref{ass:prob_perm_invariant} imposes strong conditions on the allowed distributions. While homogeneous models in which intra-community connection probability is $\pin$ and inter-community connection probability is $\pout$ satisfy the assumption, the presented proof technique does not extend to more general settings.

\section{Essentials of Poisson point processes}
\label{sec:ppp}

Denote the space of all locally finite measures on $\torus$ by $\bN$. We first provide the univariate and the bivariate Mecke equations which are used in \eqref{eq:meanZ} and \eqref{eq:Palm_second_moment} of Section \ref{sec:info_metric} respectively.
\begin{thm}[Mecke equation]
Let $0<\la<\infty$ and $\eta$ be a point process of intensity $\la$ on $\torus$. Then $\eta$ is a Poisson point process if and only if 
\[
\E\Big[\sum_{u\in \eta}f(u,\eta)\Big] = \la \int \E[f(x,\eta\cup \{u\})] \,dx= \la \int \E^{u}[f(x,\eta)] \, dx,
\]
for all measurable functions $f$ defined on $\torus \times \bN$.
\label{thm:uni_mecke}
\end{thm}
\begin{thm}[Bivariate Mecke equation]
Let $\eta$ be a Poisson process on $\torus$ with intensity $\la$. Then for every measurable function on $\torus^2 \times \bN$,
\[
\E\bigg[\sum_{u\neq u'}f(u,u',\eta)\bigg] = \la^2 \int \int  \E\Big[f(u,u',\eta\cup\{u,u'\}) \Big] \, du \, du'= \la^2 \int \int  \E^{u,u'}\Big[f(u,u',\eta) \Big]\, du \,du'.
\]
\label{thm:multi_mecke}
\end{thm}
For additional explanation about these theorems, the reader is referred to \cite[Chapter 9]{last2017lectures} and \cite[Chapter 6]{baccelli2020random}.
\end{appendices}

\end{document}